\documentclass[12pt,a4paper]{amsart}
\usepackage{amsmath,amstext,amssymb,amsthm,amsfonts,bbold}
\newcommand{\nc}{\newcommand}

\nc{\one}{\mbox{\bf 1}}
\nc{\invtensor}{\underset{\leftarrow}{\otimes}}
\nc{\const}{\operatorname{const}}

\nc{\ad}{\operatorname{ad}}
\nc{\tr}{\operatorname{tr}}

\nc{\tp}{\operatorname{top}}
\nc{\rank}{\operatorname{rank}}
\nc{\corank}{\operatorname{corank}}
\nc{\codim}{\operatorname{codim}}
\nc{\sdim}{\operatorname{sdim}}
\nc{\mult}{\operatorname{mult}}

\nc{\spn}{\operatorname{span}}
\nc{\Sym}{\operatorname{Sym}}
\nc{\sym}{\operatorname{sym}}
\nc{\id}{\operatorname{id}}
\nc{\Id}{\operatorname{Id}}
\nc{\Ree}{\operatorname{Re}}

\nc{\htt}{\operatorname{ht}}
\nc{\sch}{\operatorname{sch}}
\nc{\str}{\operatorname{str}}

\nc{\Ker}{\operatorname{Ker}}
\nc{\rker}{\operatorname{rKer}}
\nc{\im}{\operatorname{Im}}
\nc{\osp}{\mathfrak{osp}}
\nc{\sgn}{\operatorname{sgn}}
\nc{\F}{\operatorname{F}}
\nc{\Mod}{\operatorname{Mod}}
\nc{\Mat}{\operatorname{Mat}}
\nc{\Soc}{\operatorname{Soc}}
\nc{\Inj}{\operatorname{Inj}}
\nc{\Hom}{\operatorname{Hom}}
\nc{\End}{\operatorname{End}}
\nc{\supp}{\operatorname{supp}}
\nc{\Card}{\operatorname{Card}}
\nc{\Ann}{\operatorname{Ann}}
\nc{\Ind}{\operatorname{Ind}}
\nc{\Coind}{\operatorname{Coind}}
\nc{\wt}{\operatorname{wt}}
\nc{\ch}{\operatorname{ch}}
\nc{\Stab}{\operatorname{Stab}}
\nc{\Sch}{{\mathcal S}\mbox{\em ch}}
\nc{\Irr}{\operatorname{Irr}}
\nc{\Spec}{\operatorname{Spec}}
\nc{\Prim}{\operatorname{Prim}}
\nc{\Aut}{\operatorname{Aut}}
\nc{\Ext}{\operatorname{Ext}}

\nc{\Fract}{\operatorname{Fract}}
\nc{\gr}{\operatorname{gr}}
\nc{\deff}{\operatorname{def}}
\nc{\HC}{\operatorname{HC}}

\nc{\red}{\operatorname{red}}

\nc{\wdchi}{\widetilde{\chi}}
\nc{\wdH}{\widetilde{H}}
\nc{\wdN}{\widetilde{N}}
\nc{\wdM}{\widetilde{M}}
\nc{\wdO}{\widetilde{O}}
\nc{\wdR}{\widetilde{R}}
\nc{\wdS}{\widetilde{S}}
\nc{\wdV}{\widetilde{V}}

\nc{\wdC}{\widetilde{C}}

\nc{\Obj}{\operatorname{Obj}}
\nc{\Dglie}{\operatorname{{\mathcal D}glie}}
\nc{\Fin}{\operatorname{{\mathcal F}in}}

\nc{\Adm}{\operatorname{\mathcal{A}dm}}

\nc{\Sg}{{\cS(\fg)}}
\nc{\Shg}{{\cS(\fhg)}}
\nc{\Ug}{{\cU(\fg)}}
\nc{\Uhg}{{\cU(\fhg)}}
\nc{\Sh}{{\cS(\fh)}}
\nc{\Uh}{{\cU(\fh)}}
\nc{\Uhh}{{\cU(\fhh)}}
\nc{\Zg}{{{\mathcal{Z}}(\fg)}}

\nc{\Vir}{{\mathcal{V}ir}}
\nc{\NS}{{\mathcal{N}S}}

\nc{\tZg}{{\widetilde{\mathcal Z}({\mathfrak g})}}
\nc{\Zk}{{\mathcal Z}({\mathfrak k})}

\nc{\Up}{{\mathcal U}({\mathfrak p})}
\nc{\Ah}{{\mathcal A}({\mathfrak h})}
\nc{\Ag}{{\mathcal A}({\mathfrak g})}
\nc{\Ap}{{\mathcal A}({\mathfrak p})}
\nc{\Zp}{{\mathcal Z}({\mathfrak p})}
\nc{\cR}{\mathcal R}
\nc{\cS}{\mathcal S}
\nc{\cT}{\mathcal{T}}
\nc{\cY}{\mathcal Y}
\nc{\cA}{\mathcal A}
\nc{\cU}{\mathcal U}
\nc{\cH}{\mathcal H}
\nc{\cM}{\mathcal M}
\nc{\cL}{\mathcal L}
\nc{\cF}{\mathcal F}
\nc{\fg}{\mathfrak g}

\nc{\fo}{\mathfrak o}

\nc{\CO}{\mathcal O}
\nc{\CR}{\mathcal R}
\nc{\Cl}{\mathcal {C}\ell}

\nc{\cW}{\mathcal{W}}
\nc{\bM}{\mathbf{M}}
\nc{\bL}{\mathbf{L}}
\nc{\bN}{\mathbf{N}}

\nc{\zq}{\mathpzc q}

\nc{\fl}{\mathfrak l}
\nc{\fn}{\mathfrak n}
\nc{\fm}{\mathfrak m}
\nc{\fp}{\mathfrak p}
\nc{\fh}{\mathfrak h}
\nc{\ft}{\mathfrak t}
\nc{\fk}{\mathfrak k}
\nc{\fb}{\mathfrak b}
\nc{\fs}{\mathfrak s}
\nc{\fB}{\mathfrak B}

\nc{\vareps}{\varepsilon}
\nc{\varesp}{\varepsilon}
\nc{\veps}{\varepsilon}

\nc{\fsl}{\mathfrak{sl}}
\nc{\fgl}{\mathfrak{gl}}
\nc{\fso}{\mathfrak{so}}

\nc{\fpq}{\mathfrak{pq}}
\nc{\fq}{\mathfrak q}
\nc{\fsq}{\mathfrak{sq}}
\nc{\fpsq}{\mathfrak{psq}}

\nc{\fhg}{\hat{\fg}}
\nc{\fhn}{\hat{\fn}}
\nc{\fhh}{\hat{\fh}}
\nc{\fhb}{\hat{\fb}}
\nc{\hrho}{\hat{\rho}}

\nc{\hsl}{\hat{\fsl}}
\nc{\fpo}{\mathfrak{po}}
\nc{\dirlim}{\underset{\rightarrow}{\lim}\,}
\nc{\nen}{\newenvironment}
\nc{\ol}{\overline}
\nc{\ul}{\underline}
\nc{\ra}{\rightarrow}
\nc{\lra}{\longrightarrow}
\nc{\Lra}{\Longrightarrow}
\nc{\bo}{\bar{1}}
\nc{\Lla}{\Longleftarrow}

\nc{\Llra}{\Longleftrightarrow}

\nc{\thla}{\twoheadleftarrow}

\nc{\lang}{(}
\nc{\rang}{)}

\nc{\hra}{\hookrightarrow}

\nc{\iso}{\overset{\sim}{\lra}}

\nc{\ssubset}{\underset{\not=}{\subset}}

\nc{\vac}{|0\rang}

\nc{\Thm}[1]{Theorem~\ref{#1}}
\nc{\Prop}[1]{Proposition~\ref{#1}}
\nc{\Lem}[1]{Lemma~\ref{#1}}
\nc{\Cor}[1]{Corollary~\ref{#1}}
\nc{\Conj}[1]{Conjecture~\ref{#1}}
\nc{\Claim}[1]{Claim~\ref{#1}}
\nc{\Defn}[1]{Definition~\ref{#1}}
\nc{\Exa}[1]{Example~\ref{#1}}
\nc{\Rem}[1]{Remark~\ref{#1}}
\nc{\Note}[1]{Note~\ref{#1}}
\nc{\Quest}[1]{Question~\ref{#1}}
\nc{\Hyp}[1]{Hypoth\`ese~\ref{#1}}
\nen{thm}[1]{\label{#1}{\bf Theorem.\ } \em}{}
\nen{prop}[1]{\label{#1}{\bf Proposition.\ } \em}{}
\nen{lem}[1]{\label{#1}{\bf Lemma.\ } \em}{}
\nen{cor}[1]{\label{#1}{\bf Corollary.\ } \em}{}
\nen{conj}[1]{\label{#1}{\bf Conjecture.\ } \em}{}

\nen{claim}[1]{\label{#1}{\bf Claim.\ } \em}{}

\nen{defn}[1]{\label{#1}{\bf Definition.\ } }{}
\nen{exa}[1]{\label{#1}{\bf Example.\ } }{}
\nen{rem}[1]{\label{#1}{\em Remark.\ } }{}
\nen{note}[1]{\label{#1}{\em Note.\ } }{}
\nen{exer}[1]{\label{#1}{\em Exercise.\ } }{}
\nen{sket}[1]{\label{#1}{\em Sketch of proof.\ } }{}
\nen{quest}[1]{\label{#1}{\bf Question.\ } \em}{}

\nen{hyp}[1]{\label{#1}{\bf Hypoth\`ese.\ } \em}{}
\begin{document}

\setcounter{section}{0}
\setcounter{tocdepth}{1}
\date{December 2014}
\title[Affine generalized root systems]
{On the correspondence of Affine generalized root systems and symmetrizable affine Kac-Moody superalgebras}

\author{Ary Shaviv}

\address[]{Dept. of Mathematics, The Weizmann Institute of Science,
Rehovot 76100, Israel}
\email{ary.shaviv@weizmann.ac.il}

\maketitle

\begin{center}
\textsc{ Abstract}
\end{center}
Generalized root systems (GRS), that were introduced by V. Serganova, are a generalization of finite root systems (RS). We define a generalization of affine root systems (ARS), which we call {\em affine generalized root systems} (AGRS). The set of real roots of almost every symmetrizable affine indecomposable Kac-Moody superalgebra is an irreducible AGRS. In this paper we classify all AGRSs and show that almost every irreducible AGRS is the set of real roots of a symmetrizable affine indecomposable Kac-Moody superalgebra.


\tableofcontents

\newpage
\section{Introduction}
Towards the end of the $19^{th}$ century E.J. Cartan and W. Killing (see \cite{C1} and \cite{Killing1}) classified all finite dimensional simple complex Lie algebras, and showed their correspondence to finite root systems in Euclidean spaces. In the beginning of the 60's of the previous century J.P. Serre showed (see \cite{Serre1}) that any finite root system in a complex space with a non-degenerate bilinear form may be embedded in a Euclidean space, i.e. may be realized as a finite root system in a real space with a positive definite bilinear form. This realization explained why Euclidean root systems exhaust all finite dimensional simple {\em complex} Lie algebras.

\par In 1967 a full classification of Kac-Moody Lie algebras of finite and affine type was done independently by V.G. Kac and R.V. Moody (see \cite{K2} and \cite{moody}), by analyzing the Cartan matrices corresponding to the algebras. In the beginning of the 70's I.G. Macdonald classified all affine root systems (see \cite{M1}). Later in the 70's Kac realized Macdonald's affine root systems as the sets of real roots of the affine Kac-Moody algebras.

\par Kac also classified in the 70's all finite dimensional simple Lie superalgebras over an algebraically closed field of characteristic zero (see \cite{K1}). These include basic classical Lie superalgebras, which are finite dimensional simple Lie superalgebras $\fg=\fg_{\ol{0}}\oplus\fg_{\ol{1}}$ such that $\fg_{\ol{0}}$ is reductive and $\fg$ admits a consistent (i.e. even), non-degenerate, invariant, symmetric bilinear form. Approximately 20 years later, in 1996, V. Serganova classified generalized root systems (see \cite{VGRS}) and showed a correspondence between these systems and the basic classical Lie superalgebras.

\par In this paper we introduce affine generalized root systems (AGRSs) and show their correspondence to the symmetrizable affine Kac-Moody superalgebras.

\par The set of real roots of any indecomposable symmetrizable affine Kac-Moody superalgebra other than $gl(n|n)^{(1)},~n\geq1$ forms an irreducible AGRS. One might ask the converse question: is it true that an irreducible AGRS is (up to isomorphism) the set of real roots of an indecomposable symmetrizable affine Kac-Moody superalgebra?

\par We prove that the answer is no, but that there are only three exceptions: one family of finite AGRSs that are the root systems of the Kac-Moody superalgebras $gl(n|n),n\geq2$, all rational quotients of the set of real roots of the affine Kac-Moody superalgebra $gl(2|2)^{(1)}$, and all infinite quotients of the set of real roots of the affine Kac-Moody superalgebras $gl(n|n)^{(1)},~n\geq3$ (see \ref{quodef} and \ref{quodef1} below). Namely, we prove Theorem \ref{thm1}:

\subsection{}\label{thm1}
\begin{thm}{thm1}
The set of all irreducible AGRSs is a disjoint union of the following sets:
\par (1) The sets of real roots of all indecomposable symmetrizable affine Kac-Moody superalgebras other than $gl(n|n)^{(1)},n\geq1$.
\par (2) The sets of roots of the Kac-Moody superalgebras $gl(n|n),~n\geq2$.
\par (3) The rational quotients of the set of real roots of the indecomposable symmetrizable affine Kac-Moody superalgebra $gl(2|2)^{(1)}$.
\par (4) The infinite quotients of the sets of real roots of all indecomposable symmetrizable affine Kac-Moody superalgebras of the form $gl(n|n)^{(1)},~n\geq3$.
\end{thm}

\subsection{Contents of this paper}
In Section \ref{prelim} we establish the basic notation and terminology. In Section \ref{weakclasssection} we discuss the classification of all irreducible weak generalized root systems, which will be useful in understanding the structure of affine generalized root systems. In Sections \ref{generalfacts}-\ref{pfofthm} we develop the structure theory of AGRSs and prove Theorem \ref{thm1}. In Section \ref{coranaly} we summarize all known correspondences between types of root systems and Lie structures. In Section \ref{macrem} we discuss Macdonald's approach, and show how it is consistent with ours. In Section \ref{simple} we discuss general AGRSs and prove Theorem \ref{finunion} about their decomposition into irreducible ones. Finally, in Section \ref{further} we discuss the notions of parity of roots and root subsystems, give some more remarks about AGRSs, and briefly mention possible generalizations that may be made.

\section{Preliminaries}\label{prelim}

\par Throughout this paper the base field is $\mathbb{C}$ (or any other algebraically closed field of characteristic zero). All indices are assumed to be distinct unless otherwise is stated.
\subsection{}
Let $V$ be a finite dimensional complex vector space with a symmetric bilinear form $(-,-)$. Let $\widetilde{cl}$ be the canonical quotient map $\tilde{cl}: V\to V/ker(-,-)$. Note that the map $\widetilde{cl}$ induces a symmetric non-degenerate bilinear form $(-,-)$ on $V/ker(-,-)$, defined by $(\alpha,\beta):=(\widetilde{cl}^{-1}(\alpha),\widetilde{cl}^{-1}(\beta))$.

\par In what follows we introduce different types of root systems $R\subset V$ (RS, ARS, GRS, weak GRS, and AGRS). The vectors in $R$ are called {\em roots}, and a root $\alpha\in R$ is called {\em isotropic} if $(\alpha,\alpha)=0$. A system (of any type) $R$ is called {\em irreducible} if there are no two systems (of any type) $R^{1},R^{2}$, such that for all $\alpha_{1}\in R^{1},\alpha_{2}\in R^{2}:(\alpha_{1},\alpha_{2})=0$ and $R=R^{1}\sqcup R^{2}$. We call a bijective linear map $\psi: V\to V'$  an {\em isomorphism} of the systems $R\subset V,\ R'\subset V'$ if $\psi(R)=R'$ and there exists $a\in\mathbb{C}^*$ such that $(\psi(v),\psi(w))=a(v,w)$ for all $v,w\in V$. It will be clear that every system is a disjoint union of irreducible ones, and that this decomposition is unique. We call two systems $R,R'$ {\em similar} if there exists a decomposition $R=\sqcup_{i\in I}S_{i}$ and a decomposition $R'=\sqcup_{i\in I}S'_{i}$ such that $S_{i}\cong S'_{i}$ for all $i\in I$. Finally, when working with a system $R\subset V$ we denote by $cl$ the restriction of the map $\widetilde{cl}$ to $R$.

\par In \cite{VGRS} V. Serganova introduced generalized root systems, which are finite sets satisfying the conditions (0)-(2) below. In Section 7 of \cite{VGRS} Serganova introduced two alternative definitions. In the first, which we will call a weak generalized root system, the condition (2) is replaced by (2'). Both in the definition of a GRS and of a weak GRS it is assumed that $(-,-)$ is non-degenerate.

\subsection{Finite root systems}
Assume $(-,-)$ is non-degenerate.

\subsubsection{}
\begin{defn}{defroot1}
A finite non-empty set $R\subset V$ is called a {\em generalized root system} (GRS) if
the following conditions are fulfilled:

(0) $0\not\in R$ and $R$ spans $V$;

(1) if $\alpha,\beta\in R$ and $(\alpha,\alpha)\neq0$, then
$\frac{2(\alpha,\beta)}{(\alpha,\alpha)}\in \mathbb{Z}$
and $r_{\alpha}(\beta):=\beta-\frac{2(\alpha,\beta)}{(\alpha,\alpha)}\alpha\in R$;

(2) if $\alpha\in R$ and $(\alpha,\alpha)=0$, then
there exists an invertible mapping $r_{\alpha}: R\to R$ such that

$$r_{\alpha}(\beta) = \begin{cases} \mp\alpha, & \mbox{if } \beta=\pm\alpha \\ \beta, & \mbox{if } (\alpha,\beta)=0 \mbox{ and } \beta\neq\pm\alpha \end{cases}$$

$$r_{\alpha}(\beta)\in\{\beta\pm\alpha\} \mbox{ if }(\alpha,\beta)\neq0.$$
\end{defn}

\subsubsection{}
\begin{defn}{}
A GRS that contains no isotropic roots is called a {\em root system} (RS).
\end{defn}

\subsubsection{}
\begin{defn}{weak}
A finite non-empty set $R\subset V$ is called a {\em weak GRS} if it satisfies conditions (0), (1) from above and the following condition (2'):

(2') $R=-R$ and if  $\alpha\in R$ and $(\alpha,\alpha)=0$, then for any $\beta\in R$ such that $(\alpha,\beta)\neq0$, at least one of the vectors $\beta\pm\alpha$ lies in $R$.
\end{defn}

\subsection{Infinite root systems}
\subsubsection{}
\begin{defn}{affinegrs}
A (possibly infinite) non-empty set $R\subset V$ is called an {\em affine generalized root system} (AGRS) if
it satisfies the following conditions:

(0') $dim(ker(-,-))=1$, $R\cap ker(-,-)=\emptyset$, and $R$ spans $V$;

(1) if $\alpha,\beta\in R$ and $(\alpha,\alpha)\neq0$, then
$\frac{2(\alpha,\beta)}{(\alpha,\alpha)}\in \mathbb{Z}$
and $r_{\alpha}(\beta):=\beta-\frac{2(\alpha,\beta)}{(\alpha,\alpha)}\alpha\in R$;

(2) if $\alpha,\beta\in R$ and $(\alpha,\alpha)=0$, then
there exists an invertible mapping $r_{\alpha}: R\to R$ such that

$$r_{\alpha}(\beta) = \begin{cases} \mp\alpha, & \mbox{if } \beta=\pm\alpha \\ \beta, & \mbox{if } (\alpha,\beta)=0 \mbox{ and } \beta\neq\pm\alpha \end{cases}$$

$$r_{\alpha}(\beta)\in\{\beta\pm\alpha\} \mbox{ if }(\alpha,\beta)\neq0;$$

(3) $|cl(R)|<\infty$;

(4) $\forall\alpha\in cl(R),\exists\delta_{\alpha}\in ker(-,-):\forall\alpha',\alpha''\in R:cl(\alpha')=cl(\alpha'')=\alpha\Longrightarrow\alpha''-\alpha'\in\mathbb{Z}\delta_{\alpha}$.
\end{defn}

\subsubsection{}
\begin{defn}{}
An AGRS that contains no isotropic roots is called an {\em affine root system} (ARS).
\end{defn}

\par A GRS or an AGRS $R$ is called {\em reduced} if for all $\alpha\in R$: $\mathbb{C}\alpha\cap R=\{\pm\alpha\}$.

\par By~\cite{VGRS}, the map defined in axiom (2) of a GRS (which is called {\em an odd reflection}) is an involution, and it is uniquely defined, i.e. $(\alpha,\alpha)=0,(\alpha,\beta)\neq0$ implies $|\{\beta\pm\alpha\}\cap R|=1$ (see Lemma 1.11 in \cite{VGRS}). The same proof holds also for AGRSs, so in a GRS and in an AGRS the map $r_{\alpha}:R\to R$ is uniquely defined also for all isotropic roots $\alpha$.

\subsection{}
\begin{defn}{}
Let $R$ be a (possibly weak) GRS or an AGRS. We call the group generated by all well defined reflections $\widetilde{W}:=<r_{\alpha}>$ {\em the generalized Weyl group of} $R$.
\end{defn}

\begin{rem}{}
$\widetilde{W}$ is a subgroup of $Aut(R)$. Except for the cases when $R$ is a weak GRS and not a GRS, $r_{\alpha}$ is well defined for all $\alpha\in R$. For any non-isotropic $\alpha\in R$ one may define $r_{\alpha}\in End(V)$, {\em a Euclidean reflection with respect to $\alpha$}, by $r_{\alpha}(\beta):=\beta-\frac{2(\alpha,\beta)}{(\alpha,\alpha)}\alpha$. Doing so one has that if $R$ is either an RS or an ARS, then $\widetilde{W}$ is the restriction of the subgroup of $End(V)$ generated by all Euclidean reflections with respect to non-isotropic roots, and it coincides with the standard Weyl group.
\end{rem}

\subsection{}
\begin{defn}{stronglygen}
Let $R$ be a GRS and let $S\subseteq R$. Define for every $n\in\mathbb{N}$: $$S_{0}:=S~,~S_{n}:=\bigcup_{\alpha,\beta\in\cup_{j=0}^{n-1}S_{j}}r_{\alpha}(\beta).$$
We say that the subset $S$ {\em generates} $R$ if $R=\bigcup_{i=0}^{\infty}S_{i}$.
\end{defn}

\subsection{}\label{peculiar}
\begin{defn}{peculiar}
Let $q\in\mathbb{C}^{*},0\leq Re(q)<1$, and let $V$ be the three dimensional complex space spanned by the basis $\{\epsilon_{1},\delta_{1},\delta\}$ with the symmetric form: $$(\epsilon_{1},\epsilon_{1})=-(\delta_{1},\delta_{1})=1,(\epsilon_{1},\delta_{1})=(\epsilon_{1},\delta)=(\delta_{1},\delta)=(\delta,\delta)=0.$$ We call the following set $C(1,1)^{q}$:
$$C(1,1)^{q}:=\{\pm(2\epsilon_{1}+\mathbb{Z}\delta);\pm(2\delta_{1}+(q+\mathbb{Z})\delta);$$ $$\pm(\epsilon_{1}+\delta_{1}+(q+\mathbb{Z})\delta);\pm(\epsilon_{1}+\delta_{1}+\mathbb{Z}\delta);\pm(\epsilon_{1}-\delta_{1}-(q+\mathbb{Z})\delta);\pm(\epsilon_{1}-\delta_{1}+\mathbb{Z}\delta)\}.$$
We note that if $q\in\mathbb{Q}$ then $C(1,1)^{q}$ is an AGRS, and call an AGRS {\em peculiar} if it is isomorphic to $C(1,1)^{q}$ for some rational $0<q<1$.
\end{defn}

\subsection{}\label{Ann}
\begin{defn}{}
Let $n>0$ and let $V$ be the $2n+2$ dimensional space spanned by the basis $\{\epsilon_{i},\delta_{i}\}_{i\in\{1,2,..,.n+1\}}$, with the bilinear form $(\epsilon_{i},\epsilon_{j})=-(\delta_{i},\delta_{j})=\delta_{i,j},(\epsilon_{i},\delta_{j})=0$. We define $$\tilde{A}(n,n):=\{\epsilon_{i}-\epsilon_{j},\delta_{i}-\delta_{j}\}_{i\neq j}^{i,j=1,2,..n+1}\cup\{\pm(\epsilon_{i}-\delta_{j})\}_{i=1,2,..,n+1}^{j=1,2,..,n+1},$$ and $V':=span_{\mathbb{C}}\{R\}$ (a $2n+1$ dimensional subspace of $V$).
\end{defn}

\begin{rem}{}
For every $n\geq1~\tilde{A}(n,n)\subset V'$ is an irreducible finite AGRS (in Lemma \ref{lem2} we shall show that these are the only irreducible finite AGRSs). Note that $cl(\tilde{A}(1,1))\cong C(1,1)=:A(1,1)$ (see (16) in Section \ref{noisoclass} below) and if $n>1$ then $cl(\tilde{A}(n,n))=A(n,n)$ and $cl:\tilde{A}(n,n)\to A(n,n)$ is a bijection.
\par Note that $\tilde{A}(n,n)$ is the set of roots of the Lie superalgebra $gl(n+1|n+1)$.
\end{rem}

\subsection{}\label{quodef}
\begin{defn}{}
Retain the notation of Definition \ref{Ann}, and let $V'':=V\oplus span_{\mathbb{C}}\{\delta\}$, a $2n+3$ dimensional space with the form defined by the bilinear form of $V$ and $\delta$ in its kernel. We define the subset
$$\tilde{A}(n,n)^{(1)}:=\{\alpha+\mathbb{Z}\delta|~\alpha\in\tilde{A}(n,n)\}.$$
Denote $Id:=\Sigma_{i=1}^{n+1}(\epsilon_{i}-\delta_{i})$ and then $ker(-,-)=span_{\mathbb{C}}\{\delta,Id\}$.
Let us consider all possible quotients $V''/\mathbb{C}x$, where $x\in ker(-,-)\setminus\{0\}$. Under the quotient $V''/\mathbb{C}\delta$, $\tilde{A}(n,n)^{(1)}$ is mapped to the AGRS $\tilde{A}(n,n)$. All other quotients are of the form $V''/\mathbb{C}(q\delta+Id)$ for some $q\in\mathbb{C}$. We define the image of $\tilde{A}(n,n)^{(1)}$ under this quotient as $\tilde{A}(n,n)_{q}^{(1)}$. We shall see that if $n>1$ then $\tilde{A}(n,n)_{q}^{(1)}$ is an infinite AGRS with $cl(\tilde{A}(n,n)_{q}^{(1)})=A(n,n)$, and if $n=1$ this is true if and only if $q$ is rational.
\par Clearly $\tilde{A}(n,n)_{q}^{(1)}$ can be represented as:
$$\{\epsilon_{i}-\epsilon_{j}+\mathbb{Z}\delta,\delta_{i}-\delta_{j}+\mathbb{Z}\delta\}_{i\neq j}^{i,j=1,2,..n+1}\cup\{\pm(\epsilon_{i}-\delta_{j})+Id+(q+\mathbb{Z})\delta\}_{i=1,2,..,n+1}^{j=1,2,..,n+1},$$ and so $\tilde{A}(n,n)_{q}^{(1)}\cong\tilde{A}(n,n)_{q'}^{(1)}$ if $q-q'\in\mathbb{Z}$. In Lemma \ref{easyiso} we shall show that $\tilde{A}(n,n)_{q}^{(1)}\cong\tilde{A}(n,n)_{q'}^{(1)}$ if and only if either $q-q'\in\mathbb{Z}$ or $q+q'\in\mathbb{Z}$.
\par Note that $\tilde{A}(n,n)^{(1)}$ is the set of real roots of the affine Lie superalgebra $gl(n+1|n+1)^{(1)}$. We call $\tilde{A}(n,n)_{q}^{(1)}$ {\em an infinite quotient of $\tilde{A}(n,n)^{(1)}$}, and {\em a rational quotient} if $q$ is rational.
\end{defn}

\par I.G. Macdonald (see~\cite{M1}) classified all irreducible ARSs (using equivalent definitions to ours, see Section \ref{macrem} below). From the classification it follows that irreducible ARSs are precisely the sets of real roots of indecomposable affine Kac-Moody algebras, if they are reduced, and the sets of real roots of indecomposable symmetrizable affine Kac-Moody superalgebras that contain no isotropic roots and have a non-trivial odd part, if they are not reduced (see Table \ref{table1} in Section \ref{coranaly} below). V. Serganova (see~\cite{VGRS}) classified all irreducible GRSs, and from this classification it follows that these are almost precisely the root systems of the basic classical Lie superalgebras (see Table \ref{table1} in Section \ref{coranaly} and \ref{sergexception} below).

\section{Classification of weak generalized root systems}\label{weakclasssection}

Let $R$ be an irreducible weak GRS. We distinguish between two cases. The first is when all roots are non-isotropic, and the second is when $R$ contains at least one isotropic root.

\subsection{No isotropic roots}

Let $R$ be an irreducible weak GRS that contains no isotropic roots (i.e. an RS). A classical result (see for instance \cite{C1}, \cite{Killing1} and \cite{V1}) yields it must be either a root system of a finite dimensional simple Lie algebra (if it is reduced), or the non-reduced system $BC_{n}$, which we will denote by $B(0,n)$ (since it is the root system of a basic classical Lie superalgebra). The full list is: \par (1) $A_{n}$; \par (2) $B_{n}$; \par (3) $C_{n}$; \par (4) $D_{n}$; \par (5) $E_{6},E_{7},E_{8}$; \par (6) $F_{4}$; \par (7) $G_{2}$; \par (8) $BC_{n}=B(0,n)=\{\pm\delta_{i},\pm\delta_{i}\pm\delta_{j},\pm2\delta_{i}\}_{i,j\in\{1,2,...,n\}}$, when $(\delta_{i},\delta_{j})=-\delta_{i,j}$.

\subsection{$R$ contains at least one isotropic root}\label{noisoclass}

By \cite{VGRS} if $R$ contains an isotropic root, then it must be either one of the following GRSs corresponding to the basic classical Lie superalgebras (an explicit description of these is given in \cite{K1}): \par (9) $A(m,n)$, $m\geq0,n\geq1,(m,n)\neq(1,1)$; \par (10) $B(m,n)$, $m\geq1,n\geq1$; \par(11) $C(n)$, $n\geq2$; \par(12) $D(m,n)$, $m\geq2,n\geq1$; \par(13) $D(2,1;\lambda)$; \par(14) $G(3)$; \par(15) $F(4)$; \\*
or one of the two weak GRSs that are not the set of real roots of any Lie superalgebra: \par(16) $C(m,n)=\{\pm\epsilon_{i}\pm\epsilon_{j},\pm2\epsilon_{i},\pm\delta_{i}\pm\delta_{j},\pm2\delta_{i}\}_{i\neq j}\cup\{\pm\epsilon_{i}\pm\delta_{j}\}$, $m,n\geq1$; \par (17) $BC(m,n)=\{\pm\epsilon_{i}\pm\epsilon_{j},\pm\epsilon_{i},\pm2\epsilon_{i},\pm\delta_{i}\pm\delta_{j},\pm\delta_{i},\pm2\delta_{i}\}_{i\neq j}\cup\{\pm\epsilon_{i}\pm\delta_{j}\}$, $m,n\geq1$.\\* In (16),(17) $\{\epsilon_{i}\}_{i=1}^{m}\sqcup\{\delta_{j}\}_{j=1}^{n}$ is a basis of $V$ with the bilinear form $(\epsilon_{i},\epsilon_{j})=-(\delta_{i},\delta_{j})=\delta_{i,j},(\epsilon_{i},\delta_{j})=0$.

\subsubsection{}
\begin{rem}{}
In the definition of $A(m,n)$ (see \cite{K1}) the set $S:=\{\epsilon_{i},\delta_{j}\}_{i=1,2,..,m+1}^{j=1,2,..,n+1}$ is linearly independent, except for the case $A(n,n)$ for $n>1$: in that case every subset obtained from $S$ by removing one vector is linearly independent, but $\sum_{i=1}^{n+1}(\epsilon_{i}-\delta_{i})=0$ (recall that $A(n,n)$ is the root system of $psl(n+1|n+1):=sl(n+1|n+1)/<\mathbb{1}_{2n+2}>$). Hence $dim(span_{\mathbb{C}}(A(n,n)))=2n$, and in particular $A(n,n)\ncong\tilde{A}(n,n)$.
\end{rem}

\subsection{}\label{weakclass}
\begin{cor}{}
Any irreducible weak GRS has one of the forms (1)-(17).
\end{cor}

\section{Step \romannumeral1: General properties of Affine Generalized Root Systems}\label{generalfacts}

As any AGRS is a finite union of irreducible AGRSs and GRSs (see Theorem \ref{finunion} below), in order to study and understand AGRSs in general we start by studying irreducible AGRSs, and so usually assume that $R$ is an irreducible AGRS.

\subsection{Notation}
Recall that for $\alpha\in\ cl(R)$: $cl^{-1}(\alpha):=\{x\in R|~cl(x)=\alpha\}$.

\subsection{}
\begin{prop}{prop1}
Let $R$ be an AGRS in $V$, then $cl(R)$ is a weak GRS in $\widetilde{cl}(V)$. Moreover, $R$ is irreducible if and only if $cl(R)$ is irreducible.
\end{prop}
\begin{proof}
Let $R$ be an AGRS in $V$. By definition $cl(R)$ is finite (and clearly it is non-empty). Let $\alpha,\beta\in cl(R)$ and assume $(\alpha,\beta)\neq0$. We may find $\alpha',\beta'\in R$ such that $\alpha=cl(\alpha'),\beta=cl(\beta')$. Let us check the axioms of a weak GRS. \par
(0) $\alpha\neq0$ (otherwise $\alpha'\in ker(-,-)$), the rest is trivial. \par
(1) Assume $(\alpha,\alpha)\neq0$. We have $\frac{2(\alpha,\beta)}{(\alpha,\alpha)}=\frac{2(\alpha',\beta')}{(\alpha',\alpha')}\in \mathbb{Z}$. Since $r_{\alpha'}(\beta')\in R$ it is clear that $r_{\alpha}(\beta)\in cl(R)$. \par
(2) Assume $(\alpha,\alpha)=0$. One of $\alpha'\pm\beta'$ is in $R$, and so one of $\alpha\pm\beta$ is in $cl(R)$.
\par Proving that irreducibility of $R$ is equivalent to irreducibility of $cl(R)$ is straightforward.
\end{proof}

\begin{rem}{}
It may be possible to find $\alpha',\alpha'',\beta',\beta''\in R$ such that $cl(\alpha')=cl(\alpha'')=\alpha,cl(\beta')=cl(\beta'')=\beta$ and $\alpha'+\beta',\alpha''-\beta''\in R$, and so both $\alpha+\beta$ and $\alpha-\beta$ are in $cl(R)$, i.e. $cl(R)$ may sometimes be a weak GRS but not a GRS.
\end{rem}

\par Generalizing V. Kac (see ~\cite{K2}, Chapter 6) we prove the following proposition:
\subsection{}\label{lem1}
\begin{prop}{lem1}
Let $R$ be an irreducible AGRS and $\alpha,\beta\in cl(R)$ be such that $(\alpha,\beta)\not=0$.
Assume that either $(\alpha,\alpha)\not=0$ or $cl(R)\cap
\{\beta+2\alpha, \beta-2\alpha\}=\emptyset,|cl(R)\cap\{\beta+\alpha, \beta-\alpha\}|=1$.
Let $\alpha',\alpha'',\beta'\in R$ be such that $cl(\alpha')=cl(\alpha'')=\alpha,\
cl(\beta')=\beta$. Then there exists $t_{\alpha,\beta}\in\mathbb{Z}\setminus\{0\}$ such that
$$(r_{\alpha''}r_{\alpha'})^m (\beta')=\beta'+t_{\alpha,\beta} m(\alpha''-\alpha'),$$
and it is given by
$$t_{\alpha,\beta} = \begin{cases} \frac{2(\alpha,\beta)}{(\alpha,\alpha)}, & \mbox{if } (\alpha,\alpha)\neq0 \\ -1, & \mbox{if } (\alpha,\alpha)=0,\beta+\alpha\in cl(R) \\ 1, & \mbox{if } (\alpha,\alpha)=0,\beta-\alpha\in cl(R). \end{cases}$$
\end{prop}
\begin{proof}
We prove the proposition by induction on $m$. The case $m=0$ is trivial.
\par If $(\alpha,\alpha)\neq0$ then also $(\alpha',\alpha'),(\alpha'',\alpha'')\neq0$ and so $r_{\alpha'},r_{\alpha''}\in End(V)$. Therefor:
$$(r_{\alpha''}r_{\alpha'})^{m+1}(\beta')=(r_{\alpha''}r_{\alpha'})(\beta'+t_{\alpha,\beta} m(\alpha''-\alpha'))=$$ $$(r_{\alpha''}r_{\alpha'})(\beta')+t_{\alpha,\beta} m(\alpha''-\alpha')=\beta'+\frac{2(\alpha,\beta)}{(\alpha,\alpha)}(\alpha''-\alpha')+t_{\alpha,\beta} m(\alpha''-\alpha')=\beta'+t_{\alpha,\beta}(m+1)(\alpha''-\alpha')$$ (the first equality is made by induction hypothesis, the second by linearity of $r_{\alpha}$ and the fact that $\alpha''-\alpha'\in ker(-,-)$, and the third is just a straight forward calculation).

\par If $(\alpha,\alpha)=0$, then assume $r_{\alpha'}(\beta')=\beta'+\alpha'\in R$ and so $\beta+\alpha\in cl(R)$. \\* As $|cl(R)\cap\{\beta+\alpha, \beta-\alpha\}|=1$, $\beta-\alpha\not\in cl(R)$. So we have: $$(r_{\alpha''}r_{\alpha'})^{m+1}(\beta')=(r_{\alpha''}r_{\alpha'})(\beta'-m(\alpha''-\alpha'))=$$ $$r_{\alpha''}(\beta'+t_{\alpha,\beta} m(\alpha''-\alpha')+\alpha')=\beta'-(m+1)(\alpha''-\alpha'),$$ as required (the second equality is since $(\alpha',\beta'+t_{\alpha',\beta'} m(\alpha''-\alpha'))\neq0$ and $\beta-\alpha\not\in cl(R)$ and the third is as $cl(R)\cap\{\beta+2\alpha, \beta-2\alpha\}=\emptyset$). \\* Assuming $r_{\alpha'}(\beta')=\beta'-\alpha'\in R$ to get $t_{\alpha,\beta}=1$ is done the same.
\end{proof}
\subsubsection{}
\begin{rem}{}
One easily sees from the classification in Section \ref{weakclasssection} that the only case when $(\alpha,\alpha)=0$ but $cl(R)$ contains elements of the form $\beta\pm2\alpha$ (resp. $|cl(R)\cap\{\beta+\alpha, \beta-\alpha\}|=2$) is when $\alpha\in\{\pm\epsilon_{i}\pm\delta_{j}\},\beta\in\{\pm2\epsilon_{i},\pm2\delta_{j}\}$ ($\alpha\in\{\pm(\epsilon_{i}+\delta_{j})\},\beta\in\{\pm(\epsilon_{i}-\delta_{j})\}$, or vise versa).
\end{rem}

\subsection{}\label{lem2}
\begin{lem}{lem2}
Any irreducible AGRS is infinite, unless it is isomorphic to $\tilde{A}(n,n)$ for some $n\geq1$.
\end{lem}
\begin{proof}
Let $R\subset V$ be an irreducible finite AGRS. Recall that $cl(R)$ is an irreducible weak GRS.
\par Assume that $cl$ is not injective, i.e. $cl(\alpha')=cl(\alpha'')=\alpha$ for some $\alpha'\neq\alpha''\in R$.

\par It is easy to verify, using Corollary \ref{weakclass}, that unless $cl(R)\cong C(1,1),(\alpha,\alpha)=0$ then there exists $\beta\in cl(R)$ such that $(\alpha,\beta)\neq0$ and $t_{\alpha,\beta}$ is defined and non-zero. By Lemma \ref{lem1} $cl^{-1}(\beta)$ is infinite, a contradiction.

\par If $cl(R)\cong C(1,1),(\alpha,\alpha)=0$ assume $\alpha=\epsilon_{1}+\delta_{1}$ (the other cases are done in the same way). We have $\alpha'\neq\alpha''\in R'$ such that $cl(\alpha')=cl(\alpha'')=\epsilon_{1}+\delta_{1}$, denote $\alpha''-\alpha'=\delta$. Choose some $\beta'\in R$ such that $cl(\beta')=\epsilon_{1}-\delta_{1}$. One has that one of $\alpha'\pm\beta'\in R$ and so $r_{\alpha'\pm\beta'}(\alpha'')=\mp\beta'-\delta\in R$. So we have $|cl^{-1}(\epsilon_1+\delta_{1})|=|cl^{-1}(\epsilon_1-\delta_{1})|\geq2$.
\par If $|cl^{-1}(\epsilon_1+\delta_{1})|,|cl^{-1}(\epsilon_1-\delta_{1})|=2$ we have $R\cong \tilde{A}(1,1)$. Assume that there exists $\delta\neq\delta'\in ker(-,-)$ such that $\alpha'+\delta'\in R$. Then we have:
\par $r_{\beta'}(\alpha')\in\{\alpha'+\beta',\alpha'-\beta'\}$;
\par $r_{\beta'}(\alpha'+\delta)\in\{\alpha'+\beta'+\delta,\alpha'-\beta'+\delta\}$;
\par $r_{\beta'}(\alpha'+\delta')\in\{\alpha'+\beta'+\delta',\alpha'-\beta'+\delta'\}$.
\par So at least one of $cl^{-1}(\alpha+\beta),cl^{-1}(\alpha-\beta)$ has cardinality greater then 1. As both $\alpha'+\beta'$ and $\alpha'-\beta'$ are non-isotropic we are back the previous case and $R$ is infinite.

\par We conclude that, unless $R\cong \tilde{A}(1,1)$, $cl:R\to cl(R)$ is a bijection. Let $R\ncong \tilde{A}(1,1)$.
\par Assume $cl(R)$ has the form (16) or (17) (i.e. it is a weak GRS but not a GRS). As both $\epsilon_{i}+\delta_{j}+(\epsilon_{i}-\delta_{j})$ and $\epsilon_{i}+\delta_{j}-(\epsilon_{i}-\delta_{j})$ are roots in $cl(R)$, we have both $cl^{-1}(\epsilon_{i}+\delta_{j})+cl^{-1}(\epsilon_{i}-\delta_{j})$ and $cl^{-1}(\epsilon_{i}+\delta_{j})-cl^{-1}(\epsilon_{i}-\delta_{j})$ are roots in the AGRS $R$, a contradiction. So $cl(R)$ is a GRS.
\par Assume $cl(R)$ has a set of simple roots $\{\alpha_{i}\}_{i=1}^{n}$. As $\{\alpha_{i}\}_{i=1}^{n}$ are linearly independent clearly $\{cl^{-1}(\alpha_{i})\}_{i=1}^{n}$ are linearly independent, and so $\{cl^{-1}(\alpha_{i})\}_{i=1}^{n}$ is a basis in $V$. So the map $\tilde{cl}:V\to \tilde{cl}(V)$ maps the basis $\{cl^{-1}(\alpha_{i})\}_{i=1}^{n}$ to the basis $\{\alpha_{i}\}_{i=1}^{n}$, and so it is an isomorphism of the vector spaces $V$ and $\tilde{cl}(V)$. But $dim(V)-dim(\tilde{cl}(V))=dim(ker(-,-))=1$, a contradiction. Thus $cl(R)$ has no set of simple roots.
\par From the classification one sees that the only GRS that has no set of simple roots is $A(n,n)$ with $n>1$, and so $cl(R)\cong A(n,n)$. Thus there exists an isomorphism $\phi:span_{\mathbb{C}}A(n,n)\to span_{\mathbb{C}}cl(R)$. As $cl$ is injective the following $\tilde{\phi}:span_{\mathbb{C}}\tilde{A}(n,n)\to span_{\mathbb{C}}R$ is an isomorphism of $\tilde{A}(n,n)$ and $R$:
$$\tilde{\phi}(\epsilon_{i}-\epsilon_{j})=cl^{-1}(\phi(\epsilon_{i}-\epsilon_{j})),\tilde{\phi}(\delta_{i}-\delta_{j})=cl^{-1}(\phi(\delta_{i}-\delta_{j})),\tilde{\phi}(\epsilon_{i}-\delta_{j})=cl^{-1}(\phi(\epsilon_{i}-\delta_{j})),$$
i.e. $R\cong \tilde{A}(n,n)$.
\end{proof}

\section{Step \romannumeral2: The function $k$}

\subsection{}\label{nonisostr}
\begin{lem}{nonisostr}
Let $R$ be an irreducible AGRS, $\alpha\in cl(R)$ and $(\alpha,\alpha)\neq0$. Then for any $\delta_{\alpha}$ satisfying axiom (4) there exists a unique $\tilde{k}_{\delta_{\alpha}}(\alpha)\in\mathbb{Z}_{\geq0}$ such that $$cl^{-1}(\alpha)=\{\alpha'+\mathbb{Z}\tilde{k}_{\delta_{\alpha}}(\alpha)\delta_{\alpha}\}$$ for every $\alpha'\in R$ satisfying $cl(\alpha')=\alpha$.
\end{lem}
\begin{proof}
Fix some $\delta_{\alpha}$ satisfying axiom (4).
\par If $|cl^{-1}(\alpha)|=1$ then $\tilde{k}_{\delta_{\alpha}}=0$ and the claim holds.
\par Otherwise we take some $\alpha'\in cl^{-1}(\alpha)$. Choose some $\alpha''\neq\alpha'''\in cl^{-1}(\alpha)$ such that $\alpha'''-\alpha''=k\delta_{\alpha}$ is such that $k$ is the minimal positive integer possible (as $|cl^{-1}(\alpha)|>1$ there exist such two roots, and by minimality $k$ is unique). By Proposition \ref{lem1} for all $m\in\mathbb{Z}_{\geq0}$ ($t_{\alpha,\alpha}=2$):
\par $(r_{\alpha'''}r_{\alpha''})^m (\alpha'')=\alpha''+2m(\alpha'''-\alpha'')=\alpha''+2mk\delta_{\alpha}\in R$;
\par $(r_{\alpha'''}r_{\alpha''})^m (\alpha''')=\alpha'''+2m(\alpha'''-\alpha'')=\alpha''+(2m+1)k\delta_{\alpha}\in R$;
\par $(r_{\alpha''}r_{\alpha'''})^m (\alpha'')=\alpha''+2m(\alpha''-\alpha''')=\alpha''-2mk\delta_{\alpha}\in R$;
\par $(r_{\alpha''}r_{\alpha'''})^m (\alpha''')=\alpha'''+2m(\alpha''-\alpha''')=\alpha''-(2m-1)k\delta_{\alpha}\in R$,
\\* and so $cl^{-1}(\alpha)\supseteq\{\alpha''+ks\delta_{\alpha}\}_{s\in\mathbb{Z}}$.

\par By axiom (4) we have $cl^{-1}(\alpha)\subseteq\{\alpha''+s\delta_{\alpha}\}_{s\in\mathbb{Z}}$. Assume
$cl^{-1}(\alpha)\not\subseteq\{\alpha''+ks\delta_{\alpha}\}_{s\in\mathbb{Z}}$, then there exist $s',q\in\mathbb{Z}$ and $0<q<k$ such that $\alpha''+(ks'+q)\delta_{\alpha}\in R$. But $\alpha''+ks'\delta_{\alpha}\in R$ and so $(\alpha''+(ks'+q)\delta_{\alpha})-(\alpha''+ks'\delta_{\alpha})=q\delta_{\alpha}$ (recall $0<q<k$) contradicts the minimality of $k$, and so $cl^{-1}(\alpha)\subseteq\{\alpha''+ks\delta_{\alpha}\}_{s\in\mathbb{Z}}$.

\par Altogether $cl^{-1}(\alpha)\supseteq\{\alpha''+ks\delta_{\alpha}\}_{s\in\mathbb{Z}}$ and $cl^{-1}(\alpha)\subseteq\{\alpha''+ks\delta_{\alpha}\}_{s\in\mathbb{Z}}$, and so $cl^{-1}(\alpha)=\{\alpha''+ks\delta_{\alpha}\}_{s\in\mathbb{Z}}$.

\par Finally, as $cl(\alpha')=cl(\alpha'')$ we also have $cl^{-1}(\alpha)=\{\alpha''+ks\delta_{\alpha}\}_{s\in\mathbb{Z}}=\{\alpha'+ks\delta_{\alpha}\}_{s\in\mathbb{Z}}$.
\par We denote this $k$ by $\tilde{k}_{\delta_{\alpha}}(\alpha)$.
\end{proof}

\subsection{}\label{cmn11}
\begin{lem}{cmn11}
Let $R$ be an AGRS satisfying $cl(R)\cong C(1,1)=\{\pm\epsilon'_{1},\pm\delta'_{1},\pm2\epsilon'_{1}\pm2\delta'_{1}\}$. Then either $R\cong \tilde{A}(1,1)$ or $R$ is peculiar (see \ref{peculiar}).
\end{lem}
\begin{proof}
If $R\ncong \tilde{A}(1,1)$ then by Lemma \ref{lem2} it is infinite.
\par Fix two roots $\alpha_{1}\in cl^{-1}(\epsilon'_{1}-\delta'_{1}),\alpha_{2}\in cl^{-1}(2\epsilon'_{1})$. Setting $\epsilon_{1}:=\frac{1}{2}\alpha_{2},\delta_{1}:=\frac{1}{2}\alpha_{2}-\alpha_{1}$, we obtain that $\epsilon_{1}-\delta_{1},2\epsilon_{1}\in R$. Note that $\epsilon_{1}+\delta_{1}\in R$ as well, since $-r_{2\epsilon_{1}}(\epsilon_{1}-\delta_{1})=\epsilon_{1}+\delta_{1}$.
\par Moreover, $cl^{-1}(2\delta'_{1})\subset 2\delta_{1}+ker(-,-)$. We identify $cl(R)$ with $\{\pm\epsilon_{1},\pm\delta_{1},\pm2\epsilon_{1}\pm2\delta_{1}\}$.
\par We claim that $cl^{-1}(\epsilon_{1}+\delta_{1})$ is infinite. Indeed, as $R$ is infinite at least one of $cl^{-1}(\epsilon_{1}+\delta_{1}),cl^{-1}(\epsilon_{1}-\delta_{1}),cl^{-1}(2\epsilon_{1}),cl^{-1}(2\delta_{1})$ is infinite. If $cl^{-1}(\epsilon_{1}+\delta_{1})$ is infinite there is nothing to prove. If $cl^{-1}(2\epsilon_{1})$ is infinite we take two roots $\beta'\neq\beta''\in cl^{-1}(2\epsilon_{1})$ and then for every $m\in\mathbb{N}$, by Lemma \ref{lem1}
$$(r_{\beta''}r_{\beta'})^{m}(\epsilon_{1}+\delta_{1})=\epsilon_{1}+\delta_{1}+m(\beta''-\beta')\in cl^{-1}(\epsilon_{1}+\delta_{1}),$$
and so $cl^{-1}(\epsilon_{1}+\delta_{1})$ is infinite (the case $cl^{-1}(2\delta_{1})$ is infinite is similar). If $cl^{-1}(\epsilon_{1}-\delta_{1})$ is infinite, then as $r_{\epsilon_{1}-\delta_{1}+s\delta}(\epsilon_{1}+\delta_{1})\in\{2\epsilon_{1}+s\delta,2\delta_{1}+s\delta\}$
at least one of $cl^{-1}(2\epsilon_{1}),cl^{-1}(2\delta_{1})$ is infinite as well.
\par So we know $cl^{-1}(\epsilon_{1}+\delta_{1})$ is infinite. Choose $\delta\in ker(-,-)$ such that $cl^{-1}(\epsilon'_{1}+\delta'_{1})\subset \epsilon_{1}+\delta_{1}+\mathbb{Z}\delta$.
For every $\alpha\in cl(R)$ set $X(\alpha)=\{x\in\mathbb{C}|\alpha+x\delta\in R\}$.
\par By above, $X(\epsilon_{1}+\delta_{1})\in\mathbb{Z}$.
\par For every $s\in X(\epsilon_{1}+\delta_{1})$ one has
$$r_{\epsilon_{1}-\delta_{1}}(\epsilon_{1}+\delta_{1}+s\delta)\in\{2\epsilon_{1}+s\delta,2\delta_{1}+s\delta\},$$ so $X(\epsilon_{1}+\delta_{1})\subseteq X(2\epsilon_{1})\cup X(2\delta_{1})$.
\par Moreover,
$$r_{\epsilon_{1}-\delta_{1}}(cl^{-1}(2\epsilon_{1})),r_{\epsilon_{1}-\delta_{1}}(cl^{-1}(2\delta_{1}))\subseteq cl^{-1}(\epsilon_{1}+\delta_{1}),$$
so $X(\epsilon_{1}+\delta_{1})\supseteq X(2\epsilon_{1})\cup X(2\delta_{1})$ and then $X(\epsilon_{1}+\delta_{1})=X(2\epsilon_{1})\cup X(2\delta_{1})$.
\par As $r_{\epsilon_{1}-\delta_{1}}$ is invertible
$$X(\epsilon_{1}+\delta_{1})=X(2\epsilon_{1})\sqcup X(2\delta_{1}).$$
Similarly,
$$X(\epsilon_{1}-\delta_{1})=X(2\epsilon_{1})\sqcup X(-2\delta_{1})=X(2\epsilon_{1})\sqcup -X(2\delta_{1}).$$
\par In particular, $X(2\epsilon_{1}),X(2\delta_{1}),X(\epsilon_{1}-\delta_{1})\subset\mathbb{Z}$, so $X(\alpha)\in\mathbb{Z}$ for every $\alpha\in cl(R)$.
Since $2\epsilon_{1},2\delta_{1}$ are non isotropic and as $2\epsilon_{1}\in R$, by Lemma \ref{nonisostr} one has
$$X(2\epsilon_{1})=k\mathbb{Z},X(2\delta_{1})=l+k'\mathbb{Z},$$
for some $k,k',l\in\mathbb{Z},0\leq l<k'$, and at least one of $k,k'$ is non zero. So we have
$$X(\epsilon_{1}+\delta_{1})=\{k\mathbb{Z}\}\sqcup\{l+k'\mathbb{Z}\};X(\epsilon_{1}-\delta_{1})=\{k\mathbb{Z}\}\sqcup\{-l+k'\mathbb{Z}\}.$$
As $r_{\epsilon_{1}-\delta_{1}+s'\delta}(2\epsilon_{1}+s\delta)=\epsilon_{1}+\delta_{1}+(s-s')\delta$ we have
$$X(2\epsilon_{1})-X(\epsilon_{1}-\delta_{1})\subset X(\epsilon_{1}+\delta_{1}),$$
that is
$$k\mathbb{Z}+(k\mathbb{Z}\sqcup\{l+k'\mathbb{Z}\})=k\mathbb{Z}\cup\{l+k\mathbb{Z}+k'\mathbb{Z}\}\subset k\mathbb{Z}\sqcup\{l+k'\mathbb{Z}\},$$
that is
$$l+k\mathbb{Z}+k'\mathbb{Z}\subset k\mathbb{Z}\sqcup\{l+k'\mathbb{Z}\}.$$
\par Since $k\mathbb{Z}\cap\{l+k'\mathbb{Z}\}=\emptyset$, we get $k'\mathbb{Z}\cap\{l+k\mathbb{Z}\}=\emptyset$, and so $\{l+k\mathbb{Z}\}\subset\{l+k'\mathbb{Z}\}$, that is $k'\neq0$ and $k$ is divisible by $k'$.
\par Similarly, as $r_{\epsilon_{1}-\delta_{1}+s'\delta}(2\delta_{1}+s\delta)=\epsilon_{1}+\delta_{1}+(s+s')\delta$, we have $X(2\delta_{1})+X(\epsilon_{1}-\delta_{1})\subset X(\epsilon_{1}+\delta{1})$, that is
$$\{l+k'\mathbb{Z}\}+(\{k\mathbb{Z}\sqcup\{-l+k'\mathbb{Z}\}\})\subset(\{k\mathbb{Z}\}\sqcup\{l+k'\mathbb{Z}\}).$$
So we also have $k'\mathbb{Z}\subset k\mathbb{Z}\sqcup\{l+k'\mathbb{Z}\}$, so $k'$ is divisible by $k$, and we conclude that $k=k'$.
\par As $k\mathbb{Z}\cap(l+k\mathbb{Z})=\emptyset$ we have $0<l<k$.
\par So we have
$$R=\{\pm(2\epsilon_{1}+\mathbb{Z}k\delta);\pm(2\delta_{1}+(l+\mathbb{Z}k)\delta);$$ $$\pm(\epsilon_{1}+\delta_{1}+(l+\mathbb{Z}k)\delta);\pm(\epsilon_{1}+\delta_{1}+\mathbb{Z}k\delta);\pm(\epsilon_{1}-\delta_{1}-(l+\mathbb{Z}k)\delta);\pm(\epsilon_{1}-\delta_{1}+\mathbb{Z}k\delta)\}.$$

\par Finally, apply the isomorphism $$\epsilon_{1}\mapsto\epsilon_{1};\delta_{1}\mapsto\delta_{1};\delta\mapsto\frac{\delta}{k}$$ to get the required form, when $q=\frac{l}{k}$.
\end{proof}

\subsection{}\label{c11almostends}
\begin{lem}{c11almostends}
The peculiar AGRSs are the rational quotients of $\tilde{A}(1,1)^{(1)}$.
\end{lem}
\begin{proof}
By definition $\tilde{A}(1,1)^{(1)}$ is given by
$$R=\{\pm(\epsilon_{1}-\epsilon_{2})+\mathbb{Z}\delta,\pm(\delta_{1}-\delta_{2})+\mathbb{Z}\delta\}\cup\{\pm(\epsilon_{i}-\delta_{j})+\mathbb{Z}\delta\}_{i,j\in\{1,2\}}$$
when $(\epsilon_{i},\epsilon_{j})=-(\delta_{i},\delta_{j})=\delta_{i,j},(\epsilon_{i},\delta_{j})=(\epsilon_{i},\delta)=(\delta_{i},\delta)=0$. \par A rational quotient $\tilde{A}(1,1)_{q}^{(1)}$ has the form
$$\{\pm(\epsilon_{1}-\epsilon_{2})+\mathbb{Z}\delta,\pm(\epsilon_{1}+\epsilon_{2}-2\delta_{2})+(q+\mathbb{Z})\delta\}\cup$$ $$\{\pm(\epsilon_{1}-\delta_{1})+\mathbb{Z}\delta,\pm(\epsilon_{1}-\delta_{1})+(q+\mathbb{Z})\delta,\pm(\delta_{2}-\epsilon_{1})+\mathbb{Z}\delta,\pm(\delta_{2}-\epsilon_{1})+(-q+\mathbb{Z})\delta\},$$
and by scaling the bilinear form by a factor 2 it is clear that $\tilde{A}(1,1)_{q}^{(1)}\cong C(1,1)^{q}$.
\par By taking any rational $0<q<1$ we exhaust in this manner all peculiar AGRSs.
\end{proof}

\subsection{}\label{c11ends}
\begin{cor}{c11ends}
Let $R$ be an AGRS such that $cl(R)\cong C(1,1)$, then either $R\cong \tilde{A}(1,1)$ or $R$ is a rational quotient of $\tilde{A}(1,1)^{(1)}$. Moreover, any rational quotient of $\tilde{A}(1,1)^{(1)}$ is an AGRS and $cl(\tilde{A}(1,1)_{q}^{(1)})\cong C(1,1)$.
\end{cor}

\subsection{Notation}\label{quodef1}
As the set of real roots of $gl(2|2)^{(1)}$ is $\tilde{A}(1,1)^{(1)}$ we sometimes call rational quotients of $\tilde{A}(1,1)^{(1)}$ rational quotients of $gl(2|2)^{(1)}$. Similarly, as the set of real roots of $gl(n|n)^{(1)},~n\geq3$ is $\tilde{A}(n-1,n-1)^{(1)}$ we sometimes call infinite quotients of $\tilde{A}(n-1,n-1)^{(1)}$ infinite quotients of $gl(n|n)^{(1)}$.

\subsection{}\label{lem3}
\begin{prop}{lem3}
Let $R$ be an irreducible infinite AGRS, $cl(R)\ncong C(1,1)$. For every $\alpha\in cl(R)$ we fix $\delta_{\alpha}$ satisfying axiom (4). Then for every $\alpha\in cl(R)$ there exists a unique $\tilde{k}_{\delta_{\alpha}}(\alpha)\in\mathbb{Z}_{>0}$ such that for every $\alpha'\in R$ satisfying $cl(\alpha')=\alpha$: $$cl^{-1}(\alpha)=\{\alpha'+\mathbb{Z}\tilde{k}_{\delta_{\alpha}}(\alpha)\delta_{\alpha}\}.$$
\end{prop}
\begin{proof}
As $R$ is infinite and $cl(R)$ is finite there exists $\alpha\in cl(R)$ such that $cl^{-1}(\alpha)$ is infinite. If $\alpha$ is non-isotropic, then by Lemma \ref{nonisostr} it has the required form.

\par If $\alpha$ is isotropic then it is easy to verify (recall that $cl(R)\not\cong C(1,1)$) that there exists $\beta\in cl(R)$ such that both $t_{\alpha,\beta}$ and $t_{\beta,\alpha}$ are well defined, and $$\text { either } t_{\alpha,\beta}=t_{\beta,\alpha}=1 \text { or } 2t_{\alpha,\beta}=t_{\beta,\alpha}=2.$$

Fix some $\delta_{\alpha}$ satisfying axiom (4) and take some $\beta'\in cl^{-1}(\beta)$.
We take some $\alpha'\in cl^{-1}(\alpha)$. By axiom (4) we have $cl^{-1}(\alpha)\subseteq\{\alpha'+s\delta_{\alpha}\}_{s\in\mathbb{Z}}$. Choose some $\alpha''\neq\alpha'''\in cl^{-1}(\alpha)$ such that $\alpha'''-\alpha''=k\delta_{\alpha}$, where $k$ is the minimal possible positive integer.
\par We have $r_{\alpha''}r_{\alpha'}(\beta)=\beta'+(\alpha''-\alpha')\in R$.

If $t_{\beta,\alpha}=1$ then by Proposition \ref{lem1} for all $m\in\mathbb{Z}_{\geq0}$:

\par $(r_{\beta'+(\alpha''-\alpha')}r_{\beta'})^{m}\alpha''=\alpha''+mk(\alpha''-\alpha')$;
\par $(r_{\beta'}r_{\beta'+(\alpha''-\alpha')})^{m}\alpha''=\alpha''-mk(\alpha''-\alpha')$,
\\* and so $cl^{-1}(\alpha)\supseteq\{\alpha''+ks\delta_{\alpha}\}_{s\in\mathbb{Z}}$.

If $t_{\beta,\alpha}=2$ then by Proposition \ref{lem1} for all $m\in\mathbb{Z}_{\geq0}$:

\par $(r_{\beta'+(\alpha'''-\alpha'')}r_{\beta'})^m (\alpha'')=\alpha''+2m(\alpha'''-\alpha'')=\alpha''+2mk\delta_{\alpha}\in R$;
\par $(r_{\beta'+(\alpha'''-\alpha'')}r_{\beta'})^m (\alpha''')=\alpha'''+2m(\alpha'''-\alpha'')=\alpha''+(2m+1)k\delta_{\alpha}\in R$;
\par $(r_{\beta'}r_{\beta'+(\alpha'''-\alpha'')})^m (\alpha'')=\alpha''+2m(\alpha''-\alpha''')=\alpha''-2mk\delta_{\alpha}\in R$;
\par $(r_{\beta'}r_{\beta'+(\alpha'''-\alpha'')})^m (\alpha''')=\alpha'''+2m(\alpha''-\alpha''')=\alpha''-(2m-1)k\delta_{\alpha}\in R$,
\\* and again $cl^{-1}(\alpha)\supseteq\{\alpha''+ks\delta_{\alpha}\}_{s\in\mathbb{Z}}$.

\par Assume
$cl^{-1}(\alpha)\not\subseteq\{\alpha''+ks\delta_{\alpha}\}_{s\in\mathbb{Z}}$, then there exist $s',q\in\mathbb{Z}$ and $0<q<k$ such that $\alpha''+(ks'+q)\delta_{\alpha}\in R$. But $\alpha''+ks'\delta_{\alpha}\in R$ and so $(\alpha''+(ks'+q)\delta_{\alpha})-(\alpha''+ks'\delta_{\alpha})=q\delta_{\alpha}$ (recall $0<q<k$) contradicts the minimality of $k$, and so $cl^{-1}(\alpha)\subseteq\{\alpha''+ks\delta_{\alpha}\}_{s\in\mathbb{Z}}$.

\par We conclude that $cl^{-1}(\alpha)=\{\alpha''+ks\delta_{\alpha}\}_{s\in\mathbb{Z}}$. As $cl(\alpha')=cl(\alpha'')$ we also have $cl^{-1}(\alpha)=\{\alpha''+ks\delta_{\alpha}\}_{s\in\mathbb{Z}}=\{\alpha'+ks\delta_{\alpha}\}_{s\in\mathbb{Z}}$.
\par We denote this $k$ by $\tilde{k}_{\delta_{\alpha}}(\alpha)$.

\par We saw that if $cl^{-1}(\beta)$ is infinite it has the required form, so it is left to show that for every $\beta\in cl(R)$: $cl^{-1}(\beta)$ is infinite. We already found one such $\beta$ (as $R$ is infinite and $cl(R)$ is finite). Denote it by $\alpha$.

\par Choose some $\gamma\in cl(R)$ such that $(\alpha,\gamma)\neq0$ (as $cl(R)$ is irreducible such $\gamma$ exists). If $t_{\alpha,\gamma}$ is well defined, then by Lemma \ref{lem1} $cl^{-1}(\gamma)$ is infinite.

\par Otherwise, one easily sees that there exist $i',j'\in \mathbb{N}$ such that $\alpha\in\{\pm\epsilon_{i'}\pm\delta_{j'}\},\gamma\in\{\pm\epsilon_{i'}\pm\delta_{j'},\pm2\epsilon_{i'},\pm2\delta_{j'}\}$, and either $cl(R)=C(m,n)$ or $cl(R)=BC(m,n)$. Clearly $R'=\pm\{cl^{-1}(\epsilon_{i'}+\delta_{j'})\sqcup cl^{-1}(\epsilon_{i'}-\delta_{j'})\sqcup cl^{-1}(2\epsilon_{i'})\sqcup cl^{-1}(2\delta_{j'})\}$ is an AGRS, $cl(R')=C(1,1)$, and so by Lemma \ref{cmn11} $cl^{-1}(\gamma)$ is infinite.

\par Finally, as $R$ is irreducible all pre-images $\{cl^{-1}(\alpha)\}_{\alpha\in cl(R)}$ are infinite, hence have the required form.
\end{proof}

\subsection{}\label{uniformdelta}
\begin{lem}{uniformdelta}
Let $R$ be an irreducible AGRS, then there exists $\delta\in ker(-,-)$ such that for all $\alpha',\alpha''\in R$ satisfying $cl(\alpha')=cl(\alpha'')$ one has $\alpha''-\alpha'\in\mathbb{Z}\delta$.
\end{lem}
\begin{proof}
If $cl(R)\cong C(1,1)$ then either $R\cong \tilde{A}(1,1)$ or $R$ is peculiar, and in both cases the claim holds. If $R\cong \tilde{A}(n,n)$ for some $n>1$ then $cl$ is injective and the claim also holds.
\par Assume $cl(R)\ncong C(1,1)$ and $R$ is infinite. By definition for all $\alpha\in cl(R)$ there exists $\delta_{\alpha}\in ker(-,-)$ such that for all $\alpha',\alpha''\in R$ satisfying $cl(\alpha')=cl(\alpha'')$: $\alpha''-\alpha'\in\mathbb{Z}\delta_{\alpha}$. We want to find a uniform $\delta$ for all $\alpha\in cl(R)$. Let $\{\alpha_{i}\}_{i=1}^{n}$ be a maximal set (by cardinality of $n$) in $cl(R)$ such that there exists a uniform $\tilde{\delta}\in ker(-,-)$ such that for all $\alpha',\alpha''\in R$ satisfying $cl(\alpha')=cl(\alpha'')\in\{\alpha_{i}\}_{i=1}^{n}$: $\alpha''-\alpha'\in\mathbb{Z}\tilde{\delta}$ (clearly $n\geq1$).
\par Assume $cl(R)\setminus\{\alpha_{i}\}_{i=1}^{n}\neq\emptyset$. As $cl(R)\ncong C(1,1)$ is irreducible there exists $\beta\in cl(R)\setminus\{\alpha_{i}\}_{i=1}^{n}$ and $\alpha_{i}\in\{\alpha_{i}\}_{i=1}^{n}$ such that $(\alpha_{i},\beta)\neq0$, and at least one of $t_{\alpha_{i},\beta},t_{\beta,\alpha_{i}}$ is well defined (then it is a non-zero integer).
\par Assume $t_{\beta,\alpha_{i}}$ is well defined (the other case is done exactly the same). Take some $\beta'\neq\beta''\in R$ such that $cl(\beta')=cl(\beta'')=\beta$ (by Proposition \ref{lem3} such roots exist). Then, by Lemma \ref{lem1}: $r_{\beta''}r_{\beta'}\alpha'_{i}=\alpha'_{i}+t_{\beta,\alpha_{i}}(\beta''-\beta')=\alpha'_{i}+t_{\beta,\alpha_{i}}s\delta_{\beta}$ for some $s\in\mathbb{Z}\setminus\{0\}$. On the other hand we have $t_{\beta,\alpha_{i}}s\delta_{\beta}=s'\tilde{\delta}$ for some $s'\in\mathbb{Z}\setminus\{0\}$. Altogether we got $\tilde{\delta}=\frac{t_{\beta,\alpha_{i}}s}{s'}\delta_{\beta},\delta_{\beta}=\frac{s'}{t_{\beta,\alpha_{i}}s}\tilde{\delta}$. Defining $\tilde{\tilde{\delta}}:=\frac{\delta_{\beta}}{s'}$, we get $\delta_{\beta},\tilde{\delta}\in\mathbb{Z}\tilde{\tilde{\delta}}$, and so $\tilde{\tilde{\delta}}$ is uniform for $\{\alpha_{i}\}_{i=1}^{n}\sqcup\{\beta\}$, a contradiction.
\end{proof}

\subsection{Notation}
From now on by $\delta$ we mean any $\delta\in ker(-,-)$ satisfying Lemma \ref{uniformdelta}. We will always work with such $\delta$s instead of $\delta_{\alpha}$s and omit the lower index in $\tilde{k}_{\delta}$ (and just write $\tilde{k}$, though $\tilde{k}$ depends on $\delta$). Recall that by Lemma \ref{lem3} for every $\alpha\in cl(R)\ncong C(1,1)$: $R$ is infinite implies $\tilde{k}(\alpha)\in\mathbb{Z}_{>0}$.
\par For a set $S\subseteq V$ we denote $S^{\perp}=\{v\in V|(v,S)=0\}$. In particular, as $R$ spans $V$ we have $R^{\perp}=ker(-,-)=span\{\delta\}$.

\subsection{}\label{weylinvlemma}
\begin{lem}{weylinvlemma}
Let $R$ be an irreducible infinite AGRS, $cl(R)\ncong C(1,1)$, $\alpha,\beta\in cl(R)$. Recall that $cl(R)$ is a weak GRS, and assume $r_{\alpha}$ is defined in $cl(R)$. Then $\tilde{k}(\beta)=\tilde{k}(r_{\alpha}\beta)$.
\end{lem}
\begin{proof}
Take some $\alpha',\beta'\in R$ such that $cl(\alpha')=\alpha, cl(\beta')=\beta$.
If $(\alpha,\beta)=0$ the claim is trivial, so we may assume $(\alpha,\beta)\neq0$.
\par Assume $(\alpha,\alpha)=0$, then $r_{\alpha'}(\beta')\in\{\beta'\pm\alpha'\},r_{\alpha'}(\beta'+\tilde{k}(\beta)\delta)\in\{\beta'+\tilde{k}(\beta)\delta\pm\alpha'\}$. As $r_{\alpha}$ is well defined in $cl(R)$: \\* $r_{\alpha'}(\beta'+\tilde{k}(\beta)\delta)=r_{\alpha'}(\beta')+\tilde{k}(\beta)\delta$, $r_{\alpha'}(\beta'+\tilde{k}(\beta)\delta)-r_{\alpha'}(\beta')=\tilde{k}(\beta)\delta$ (i.e. the signs are the same in both reflections). Hence we have $\tilde{k}(r_{\alpha}\beta)\leq \tilde{k}(\beta)$.
\par Assume $(\alpha,\alpha)\neq0$, then $r_{\alpha'}$ is linear. $r_{\alpha'}(\beta')=\beta'-\frac{2(\alpha',\beta')}{(\alpha',\alpha')}\alpha'$ and $r_{\alpha'}(\beta'+\tilde{k}(\beta')\delta)=\beta'+\tilde{k}(\beta')\delta-\frac{2(\alpha',\beta')}{(\alpha',\alpha')}\alpha'$ so $r_{\alpha'}(\beta'+\tilde{k}(\beta)\delta)-r_{\alpha'}(\beta')=\tilde{k}(\beta)\delta$. Hence we have $\tilde{k}(r_{\alpha}\beta)\leq \tilde{k}(\beta)$.
\par In both cases we have $\tilde{k}(r_{\alpha}\beta)\leq \tilde{k}(\beta)$. As $r_{\alpha}$ is an involution we also have $\tilde{k}(r_{\alpha}\beta)\geq \tilde{k}(\beta)$ and so $\tilde{k}(r_{\alpha}\beta)=\tilde{k}(\beta)$.
\end{proof}

\subsection{}\label{weylinvariant}
\begin{cor}{weylinvariant}
Let $R$ be a non-peculiar irreducible infinite AGRS, and let $\widetilde{W}$ be the generalized Weyl group of the weak GRS $cl(R)$. Then the function $\tilde{k}:cl(R)\to\mathbb{Z}_{>0}$ is $\widetilde{W}$ invariant.
\end{cor}

\begin{rem}{}
Assume $(\alpha,\alpha)=0$, $(\alpha,\beta)\neq0$ and $cl(R)\cap\{\beta\pm2\alpha\}=\emptyset$. Then one of $\beta\pm\alpha\in cl(R)$. In that case $\tilde{k}(\beta)=\tilde{k}(\beta\pm\alpha)$ (the proof is exactly like the one of Lemma \ref{weylinvlemma}). This is true even if $cl(R)\cap\{\beta'\pm2\alpha\}\neq\emptyset$ for some other $\beta'\in cl(R)$ (and then $r_{\alpha}$ is not defined in $cl(R)$). Thus, if $cl(R)\cong BC(m,n)$ or $cl(R)\cong C(m,n)$ one has has that $\tilde{k}$ is constant on the {\em set} $\{\pm\epsilon_{i}\pm\epsilon_{j},\pm\epsilon_{i}\pm\delta_{j},\pm\delta_{i}\pm\delta_{j}\}$ and on the set $\{\epsilon_{i},\delta_{j}\}$ (although in $BC(m,n)$ and $C(m,n)$ these sets are not $\widetilde{W}$ orbits as $r_{\epsilon_{i}\pm\delta_{j}}$ is not defined).
\end{rem}

\subsection{}\label{combinedlemma}
\begin{lem}{combinedlemma}
Let $R$ be an irreducible infinite AGRS, $cl(R)\ncong C(1,1)$.
\par (\romannumeral1) Let $\alpha,\beta\in cl(R)$ be such that either $(\alpha,\alpha)\neq0$ or $cl(R)\cap\{\beta+2\alpha,\beta-2\alpha\}=\emptyset,|cl(R)\cap\{\beta+\alpha,\beta-\alpha\}|=1$. Then $t_{\alpha,\beta}\tilde{k}(\alpha)$ is divisible by $\tilde{k}(\beta)$.
\par (\romannumeral2)
Assume either $cl(R)\cong C(m,n)$ or $cl(R)\cong BC(m,n)$. Then $\tilde{k}(2\epsilon_{i})=\tilde{k}(2\delta_{j})=2\tilde{k}(\epsilon_{i}+\delta_{j})$.
\end{lem}
\begin{proof}
Let $\alpha',\beta'\in R$ be such that $cl(\alpha')=\alpha,cl(\beta')=\beta$. By Proposition \ref{lem1} $$r_{\alpha'+\tilde{k}(\alpha)\delta}r_{\alpha'}(\beta')=\beta'+t_{\alpha,\beta}\tilde{k}(\alpha)\delta\in\{\beta'+s\tilde{k}(\beta)\delta\}_{s\in\mathbb{Z}},$$ so $t_{\alpha,\beta}\tilde{k}(\alpha)$ is divisible by $\tilde{k}(\beta)$.
\par We are left to prove (\romannumeral2). By Corollary \ref{weylinvariant} $\tilde{k}(\epsilon_{i}+\delta_{j})$ is independent of $i,j$. Fix some $i',j'$. Clearly $R'=\pm\{cl^{-1}(\epsilon_{i'}+\delta_{j'})\sqcup cl^{-1}(\epsilon_{i'}-\delta_{j'})\sqcup cl^{-1}(2\epsilon_{i'})\sqcup cl^{-1}(2\delta_{j'})\}$ is an AGRS. As $R$ is non-peculiar and $cl(R')\cong C(1,1)$, we must have that $cl^{-1}(cl(R'))\cong C(1,1)^{\frac{1}{2}}$ and so $\tilde{k}(2\epsilon_{i})=\tilde{k}(2\delta_{j})=2\tilde{k}(\epsilon_{i}+\delta_{j})$ by Lemma \ref{cmn11}, as required.
\end{proof}

\subsection{}\label{unik}
\begin{prop}{unik}
Let $R$ be a non-peculiar irreducible infinite AGRS. There exists a unique (up to a $\pm$ sign) $\delta\in ker(-,-)$ and a unique map $k:cl(R)\to\mathbb{Z}_{>0}$ such that for all $\alpha\in cl(R)$ and all $\alpha'\in cl^{-1}(\alpha)$: $cl^{-1}(\alpha)=\{\alpha'+\mathbb{Z}k(\alpha)\delta\}$, and $gcd\{k(\alpha)\}_{\alpha\in cl(R)}=1$.
\end{prop}

\begin{proof}
Choose some $\delta'\in ker(-,-)$ satisfying Lemma \ref{uniformdelta}, and define $\tilde{k}_{\delta'}$ in Lemma \ref{lem3}. Define $\delta:=gcd\{\tilde{k}_{\delta'}(\alpha)\}_{\alpha\in cl(R)}\delta'$. Clearly, $\delta$ satisfies the conditions of Lemma \ref{uniformdelta}, and $k:=\tilde{k}_{\delta}$ in Lemma \ref{lem3} satisfies $gcd\{k(\alpha)\}_{\alpha\in cl(R)}=1$. By Lemma \ref{lem3} this $k$ is unique, and $\delta$ is unique up to a $\pm$ sign.
\end{proof}

\subsection{Notation}
From now on by $\delta$ we mean any $\delta\in ker(-,-)$ satisfying Lemma \ref{unik}. As shown, it is unique up to a $\pm$ sign.

\subsection{}
\begin{rem}{Extending the function $k$}
Given a non-peculiar irreducible infinite AGRS one may define the function $k$ as a function from $R$ to $\mathbb{Z}_{>0}$ by $\alpha\mapsto cl(\alpha)\mapsto k(cl(\alpha))$. In that case, however, the function $k$ is not invariant with respect to the action of the generalized Weyl group of $R$. As an example take $2\epsilon_{i}$ and $\epsilon_{i}+\delta_{j}$ in $A(2m-1,2n-1)^{(2)}$: as $r_{\epsilon_{i}-\delta_{j}}(2\epsilon_{i})=\epsilon_{i}+\delta_{j}$ they are in the same orbit, but $k(2\epsilon_{i})=2k(\epsilon_{i}+\delta_{j})$.
\end{rem}

\section{Step \romannumeral3: All possible $k$s}
Let $R$ be a non-peculiar irreducible infinite AGRS. By Corollary \ref{weakclass}, Proposition \ref{prop1}, and Lemma \ref{lem2} $cl(R)$ has one of the forms (1)-(17). In the following we find all possible functions $k:cl(R)\to\mathbb{Z}_{>0}$ for all possible forms (1)-(17). Note that as $gcd\{k(\alpha)\}_{\alpha\in cl(R)}=1$ the set of ratios $\{\frac{k(\alpha)}{k(\beta)}\}_{\alpha,\beta\in cl(R)}$ completely determines $k$. In addition, for each possible pair $(cl(R),k:cl(R)\to\mathbb{Z}_{>0})$ we find some $\fg$, an indecomposable symmetrizable affine Kac-Moody superalgebra, such that its set of real roots $\Delta^{re}(\fg)$ is an AGRS, $cl(\Delta^{re}(\fg))\cong cl(R)$ and $k(cl(\Delta^{re}(\fg)))\equiv k(cl(R))$. We call this $\fg$ {\em a representative of $R$}. In Section \ref{pfofthm} we will show that in almost all cases $R$ is isomorphic to the set of real roots of its representative.

\begin{rem}{}
In general we follow Kac's notation (see \cite{K2} and \cite{K1}). In the cases where we have a so called {\em twisted} Lie superalgebra we follow S. Reif's notation (see \cite{R1}).
\end{rem}

\subsection{$A_{n}$} \

$cl(R)\cong A_{n}$, there is only one $\widetilde{W}$ orbit, so by Corollary \ref{weylinvariant} $k$ is constant, and $A_{n}^{(1)}$ is a representative.

\subsection{$B_{n}$} \

$cl(R)\cong B_{n}$, by Corollary \ref{weylinvariant} $k$ is constant on the orbit $\{\pm\epsilon_{i}\pm\epsilon_{j}\}$ and on the orbit $\{\pm\epsilon_{i}\}$.
\par As $|t_{\epsilon_{i},\epsilon_{i}-\epsilon_{j}}|=2|t_{\epsilon_{i}-\epsilon_{j},\epsilon_{i}}|=2$ by Lemma \ref{combinedlemma}.\romannumeral1~ we have $k(\epsilon_{i})\in\{k(\epsilon_{i}-\epsilon_{j}),\frac{1}{2}k(\epsilon_{i}-\epsilon_{j})\}$. In the first case $B_{n}^{(1)}$ is a representative, and in the second $D_{n+1}^{(2)}$ is a representative.

\subsection{$C_{n}$}\label{Cn} \

$cl(R)\cong C_{n}$, by Corollary \ref{weylinvariant} $k$ is constant on the orbit $\{\pm\epsilon_{i}\pm\epsilon_{j}\}$, and on the orbit $\{\pm2\epsilon_{i}\}$. \par As $|t_{\epsilon_{i}-\epsilon_{j},2\epsilon_{i}}|=2|t_{2\epsilon_{i},\epsilon_{i}-\epsilon_{j}}|=2$ by Lemma \ref{combinedlemma}.\romannumeral1~ we have $k(\epsilon_{i}-\epsilon_{j})\in\{k(2\epsilon_{i}),\frac{1}{2}k(2\epsilon_{i})\}$. In the first case $C_{n}^{(1)}$ is a representative, and in the second $A_{2n-1}^{(2)}$ is a representative.

\subsection{$D_{n}$} \

$cl(R)\cong D_{n}$, there is only one $\widetilde{W}$ orbit, so by Corollary \ref{weylinvariant} $k$ is constant, and $D_{n}^{(1)}$ is a representative.

\subsection{$E_{6},E_{7},E_{8}$} \

$cl(R)\cong E_{i}$ for some $i\in\{6,7,8\}$, there is only one $\widetilde{W}$ orbit, so by Corollary \ref{weylinvariant} $k$ is constant, and $E_{i}^{(1)}$ is a representative.

\subsection{$F_{4}$} \

$cl(R)\cong F_{4}$, by Corollary \ref{weylinvariant} $k$ is constant on the orbit $\{\pm\epsilon_{i},\frac{1}{2}(\pm\epsilon_{1}\pm\epsilon_{2}\pm\epsilon_{3}\pm\epsilon_{4})\}$, and on the orbit $\{\pm\epsilon_{i}\pm\epsilon_{j}\}$. \par As $t_{\epsilon_{1},\epsilon_{1}+\epsilon_{2}}=2t_{\epsilon_{1}+\epsilon_{2},\epsilon_{1}}=2$ we have, by Lemma \ref{combinedlemma}.\romannumeral1, two options: $k(\epsilon_{1})\in\{k(\epsilon_{1}+\epsilon_{2}),\frac{1}{2}k(\epsilon_{1}+\epsilon_{2})\}$. In the first case $F_{4}^{(1)}$ is a representative, and in the second $E_{6}^{(2)}$ is a representative.

\subsection{$G_{2}$} \

$cl(R)\cong G_{2}$, by Corollary \ref{weylinvariant} $k$ is constant on the orbit $\{\epsilon_{i}-\epsilon_{j}\}$ and on the orbit $\{\pm\epsilon_{i}\}$. \par As $3|t_{\epsilon_{i}-\epsilon_{j},\epsilon_{i}}|=|t_{\epsilon_{i},\epsilon_{i}-\epsilon_{j}}|=3$ by Lemma \ref{combinedlemma}.\romannumeral1~ we have $k_{\epsilon_{i}-\epsilon_{j}}\in\{k_{\epsilon_{i}},3k_{\epsilon_{i}}\}$. In the first case $G_{2}^{(1)}$ is a representative, and in the second $D_{4}^{(3)}$ is a representative.

\subsection{$B(0,n)$}\label{b0nclass1} \

$cl(R)\cong B(0,n)$. As $\{\pm\delta_{i}\pm\delta_{j},\pm2\delta_{i}\}\cong C_n$ by the analysis in \ref{Cn} we have two options.

\par (i) $k(\pm\delta_{i}\pm\delta_{j})=k(\pm2\delta_{i})$. If this is the minimal $k(\alpha)$ in $\{k(\alpha)\}_{\alpha\in cl(R)}$ then (as $|t_{\pm\delta_{i}\pm\delta_{j},\pm\delta_{i}}|=1$) it is also obtained for $k(\delta_{i})$ and so $B(0,n)^{(1)}$ is a representative. If this is not the minimal $k(\alpha)$ in $\{k(\alpha)\}_{\alpha\in cl(R)}$ then (as $|t_{\pm\delta_{i},\pm\delta_{i}\pm\delta_{j}}|=2$ and use \ref{combinedlemma}.\romannumeral1) we must have $k(\delta_{i})=\frac{1}{2}k(2\delta_{i})$ and so $C(n+1)^{(2)}$ is a representative.

\par (ii) $2k(\pm\delta_{i}\pm\delta_{j})=k(\pm2\delta_{i})$. If the minimal $k(\alpha)$ in $\{k(\alpha)\}_{\alpha\in cl(R)}$ is $k(\pm\delta_{i}\pm\delta_{j})$ then (as $|t_{\pm\delta_{i}\pm\delta_{j},\pm\delta_{i}}|=1$) we have $k(\delta_{i})=k(\pm\delta_{i}\pm\delta_{j})$ and so $A(0,2n-1)^{(2)}$ is a non-reduced representative, and $A_{2n}^{(2)}$ is a reduced representative. Otherwise the minimal $k(\alpha)$ in $\{k(\alpha)\}_{\alpha\in cl(R)}$ is $k(\pm\delta_{i})$ and then (as $|t_{\pm\delta_{i},\pm\delta_{i}\pm\delta_{j}}|=2$ and use \ref{combinedlemma}.\romannumeral1) we must have $k(\delta_{i})=\frac{1}{2}k(\pm\delta_{i}\pm\delta_{j})$ and so $A(0,2n)^{(4)}$ is a representative.

\subsection{$A(m,n)$, where $m\geq0,n\geq1,(m,n)\neq(1,1)$} \

$cl(R)\cong A(m,n)$, there is only one $\widetilde{W}$ orbit, so by Corollary \ref{weylinvariant} $k$ is constant, and $A(m,n)^{(1)}$ is a representative.

\subsection{$B(m,n)$, where $m\geq1,n\geq1$} \

$cl(R)\cong B(m,n)$. By Corollary \ref{weylinvariant} $k$ is constant on the orbit $\{\epsilon_{i},\delta_{j}\}$ and on the orbit $\{\pm\epsilon_{i}\pm\epsilon_{j},\pm\delta_{i}\pm\delta_{j},\pm\epsilon_{i}\pm\delta_{j}\}$. As $\{\pm\delta_{i}\pm\delta_{j},\pm2\delta_{i}\}\cong C_{n}$ then by the analysis in \ref{Cn} $k(\pm2\delta_{i})\in\{k(\pm\delta_{i}\pm\delta_{j}),2k(\pm\delta_{i}\pm\delta_{j})\}$. But since for isotropic $\alpha\in cl(R)$ we have $t_{\alpha,\beta}\in\{0,\pm1\}$ for all $\beta$, and in particular this is true for $\beta=\pm2\delta_{i}$: $k(\pm2\delta_{i})\leq k(\pm\epsilon_{i}\pm\delta_{j})=k(\pm\delta_{i}\pm\delta_{j})$ and so $k(\pm2\delta_{i})= k(\pm\epsilon_{i}\pm\delta_{j})$.

\par If $k(\epsilon_{i}+\delta_{j})$ is the minimal in $\{k(\alpha)\}_{\alpha\in cl(R)}$ then (as for isotropic $\alpha\in cl(R)$ we have $t_{\alpha,\beta}\in\{0,\pm1\}$ for all $\beta$) it is the $k(\alpha)$ for all $\alpha\in cl(R)$ and so $k$ is constant and $B(m,n)^{(1)}$ is a representative.

\par Otherwise, the minimal $k(\alpha)$ in $\{k(\alpha)\}_{\alpha\in cl(R)}$ is obtained for $\pm\epsilon_{i}$ (and not for isotropic roots). As $t_{\epsilon_{i},\epsilon_{i}+\delta_{j}}=2$ we must have, by \ref{combinedlemma}.\romannumeral1, that $k(\pm\epsilon_{i}\pm\epsilon_{i})=2k(\pm\epsilon_{i})$. As $k(\delta_{j})=k(\epsilon_{i})$, $D(m+1,n)^{(2)}$ is a representative.

\subsection{$C(n)$, where $n\geq2$} \

$cl(R)\cong C(n)$, there is only one $\widetilde{W}$ orbit, so by Corollary \ref{weylinvariant} $k$ is constant, and $C(n)^{(1)}$ is a representative.

\subsection{$D(m,n)$, where $m\geq2,n\geq1$} \

$cl(R)\cong D(m,n)$, there is only one $\widetilde{W}$ orbit, so by Corollary \ref{weylinvariant} $k$ is constant, and $D(m,n)^{(1)}$ is a representative.

\subsection{$D(2,1;\lambda)$} \

$cl(R)\cong D(2,1;\lambda)$, there is only one $\widetilde{W}$ orbit, so by Corollary \ref{weylinvariant} $k$ is constant, and $D(2,1;\lambda)^{(1)}$ is a representative.

\subsection{$G(3)$} \

$cl(R)\cong G(3)$, there is only one $\widetilde{W}$ orbit, so by Corollary \ref{weylinvariant} $k$ is constant, and $G(3)^{(1)}$ is a representative.

\subsection{$F(4)$} \

$cl(R)\cong F(4)$, there is only one $\widetilde{W}$ orbit, so by Corollary \ref{weylinvariant} $k$ is constant, and $F(4)^{(1)}$ is a representative.

\subsection{$C(m,n),(m,n)\neq(1,1)$} \

$cl(R)\cong C(m,n)$, by Corollary \ref{weylinvariant} $k$ is constant on the set $\{\pm\epsilon_{i}\pm\epsilon_{j},\pm\epsilon_{i}\pm\delta_{j},\pm\delta_{i}\pm\delta_{j}\}$.
\par By Lemma \ref{combinedlemma}.\romannumeral2 ~ $k(2\epsilon_{i})=k(2\delta_{j})=2k(\pm\epsilon_{i}\pm\epsilon_{j})$ and so $A(2m-1,2n-1)^{(2)}$ is a representative.

\subsection{$BC(m,n)$}\label{bcmn} \

$cl(R)\cong BC(m,n)$. By Corollary \ref{weylinvariant} $k$ is constant on the set $\{\pm\epsilon_{i}\pm\epsilon_{j},\pm\epsilon_{i}\pm\delta_{j},\pm\delta_{i}\pm\delta_{j}\}$, and on the set $\{\pm\epsilon_{i},\pm\delta_{j}\}$. By Lemma \ref{combinedlemma}.\romannumeral2  ~ $k(2\epsilon_{i})=k(2\delta_{j})=2k(\epsilon_{i}+\delta_{j})$.
\par
$2|t_{\pm\epsilon_{i}\pm\epsilon_{j},\pm\epsilon_{i}}|=|t_{\pm\epsilon_{i},\pm\epsilon_{i}\pm\epsilon_{j}}|=2$ so by Lemma \ref{combinedlemma}.\romannumeral1~ $k(\pm\epsilon_{i})\in\{k(\pm\epsilon_{i}\pm\epsilon_{j}),\frac{1}{2}k(\pm\epsilon_{i}\pm\epsilon_{j})\}=\{k(\epsilon_{i}+\delta_{j}),\frac{1}{2}k(\epsilon_{i}+\delta_{j})\}$.
In the first case $A(2m,2n-1)^{(2)}$ is a representative, and in the second $A(2m,2n)^{(4)}$ is a representative.

\section{Step \romannumeral4: Proof of the Theorem}\label{pfofthm}

In the following we classify all irreducible AGRSs, and as a result of the classification we prove Theorem \ref{thm1}. We are doing so by going over all possible irreducible weak GRSs $cl(R)$, and finding all possible $R$s corresponding to this class. In this section we always assume $R$ is irreducible.

\subsection{$cl(R)$ is a GRS generated by a set of simple roots} \

Let $cl(R)$ be a GRS generated by a set of simple roots, i.e. $cl(R)$ has one of the forms (1)-(7),(9) with $m\neq n$, or (10)-(15). By Lemma \ref{lem2} $R$ is infinite and its structure is given by Proposition \ref{unik}. In the following Proposition \ref{shiftlem} we show that $cl(R)$ together with the function $k:cl(R)\to\mathbb{Z}_{>0}$ completely determines $R$, and so $R$ is isomorphic to the set of real roots of its representative found above.

\subsubsection{}\label{shiftlem}
\begin{prop}{shiftlem}
Let $R\subset V,R'\subset V'$ be two infinite AGRSs, $cl(R)\cong cl(R')$ is an irreducible GRS generated by a set of simple roots, and $\phi:V/ker(-,-)\to V'/ker(-,-)$ is an isomorphism of the GRSs $cl(R)$ and $cl(R')$. Assume that for every $\alpha\in cl(R)$: $k(\alpha)=k(\phi(\alpha))$. Then $R$ is isomorphic to $R'$.
\end{prop}
\begin{proof}
We take $\{\alpha_{i}\}_{i=1}^{n}\subset cl(R)$ a set of simple roots, and fix some $\{\alpha'_{i}\}_{i=1}^{n}\subset R$ such that for all $i:cl(\alpha'_{i})=\alpha_{i}$. Clearly, the set $\{\alpha'_{i}\}_{i=1}^{n}\sqcup\{\delta\}$ is a basis of $V$. Define $V_{1}:=span_{\mathbb{C}}\{\alpha'_{i}\}_{i=1}^{n}$ and denote $cl(R)=\{\alpha_{i}\}_{i=1}^{N}$, when $N>n$.

\par As the set $\{\alpha'_{i}\}_{i=1}^{n}\subset R$ generates a copy of $cl(R)$ we have that for all $i\in\{1,2,...,n,n+1,...,N\}$ there exists $\alpha'_{i}\in V_{1}$ such that $cl(\alpha'_{i})=\alpha_{i}$. Note that these roots are obtained by simple reflections, and these reflections depend only on the classes $\{\alpha_{i}\}_{i=1}^{n}\subset cl(R)$ and not the representatives $\{\alpha'_{i}\}_{i=1}^{n}\subset R$ chosen. Thus we have for all $i\in\{1,2,...,n,n+1,...,N\}:cl^{-1}(\alpha_{i})=\{\alpha'_{i}+sk(\alpha_{i})\delta\}_{s\in\mathbb{Z}}$, and $R=\{\alpha'_{i}+sk(\alpha_{i})\delta\}_{i=1,2,...,N}^{s\in\mathbb{Z}}$.

\par If $R'\subset V'$ satisfies $\{\alpha_{i}\}_{i=1}^{N}=cl(R)\cong cl(R')=\{\phi(\alpha_{i})\}_{i=1}^{N}$, then the linear map from $V$ to $V'$ defined by $\alpha_{i}\mapsto\phi(\alpha_{i})'$ for all $i\in\{1,2,...,n\}$, $\delta\mapsto\delta'$ (note that now $R'=\{\phi(\alpha_{i})'+sk(\phi(\alpha_{i}))\delta'\}_{i=1,2,...,N}^{s\in\mathbb{Z}}$) gives an isomorphism of the AGRSs $R$ and $R'$.
\end{proof}

\subsection{$cl(R)\cong B(0,n)$}\label{bzeroncase} \

Let $cl(R)=B(0,n)$, i.e. $cl(R)$ has the form (8). We will see that now $cl(R)$ together with the function $k$ almost always determine $R$. We start by the following auxiliary Lemma \ref{B0nlcomp}:

\subsubsection{}\label{B0nlcomp}
\begin{lem}{B0nlcomp}
Let $R$ be a reduced AGRS, such that $cl(R)=B(0,n)$. Then $k(2\delta_{i})=2k(\delta_{i}-\delta_{j})=2k(\delta_{i})$.
\end{lem}
\begin{proof}
Let $\delta'_{i}\in cl^{-1}(\delta_{i})$. We have $2\delta'_{i}+l\delta\in R$ for some $l\in\mathbb{C}$.
$$r_{\delta'_{i}}(2\delta'_{i}+l\delta)=-(2\delta'_{i}-l\delta)\in R,$$ and so $2l\equiv0~(mod~k(2\delta_{i}))$. As $l$ is arbitrary up to $\mathbb{Z}k(2\delta_{i})$: $l\not\equiv0~(mod~k(2\delta_{i}))$ (otherwise also $2\delta'_{i}\in R$, a contradiction to $R$ being reduced) and so we may choose $l=\frac{k(2\delta_{i})}{2}$.
$$r_{2\delta'_{i}+\frac{k(2\delta_{i})}{2}\delta}(\delta'_{i})=-(\delta'_{i}+\frac{k(2\delta_{i})}{2}\delta),$$ and so
$$k(2\delta_{i})\neq k(\delta_{i}).$$ We also know
$$k(2\delta_{i})\neq4k(\delta_{i}),$$ otherwise we get a contradiction to $R$ being reduced again. By \ref{b0nclass1} it is left to show that if $n>1$ then $k(\delta_{i}+\delta_{j})\neq k(2\delta_{i})$. Assume $n>1$, assume $k(\delta_{i}+\delta_{j})=k(2\delta_{i})$, and take some $p\in\mathbb{C}$ such that $\delta'_{j}-\delta'_{i}+p\delta\in R$. One has:
$$r_{\delta'_{i}}(\delta'_{j}-\delta'_{i}+p\delta)=\delta'_{j}+\delta'_{i}+p\delta,r_{2\delta'_{i}+\frac{k(\delta_{i}+\delta_{j})}{2}\delta}(\delta'_{j}-\delta'_{i}+p\delta)=\delta'_{j}+\delta'_{i}+(p+\frac{k(\delta_{i}+\delta_{j})}{2})\delta,$$ a contradiction.
\end{proof}

\subsubsection{}\label{b0nfinalprop}
\begin{prop}{}
Let $R\subset V,R'\subset V'$ be two AGRSs, $cl(R)\cong cl(R')\cong B(0,n)$, when $\phi:V/ker(-,-)\to V'/ker(-,-)$ is an isomorphism of the GRSs $cl(R)$ and $cl(R')$. Assume that for every $\alpha\in cl(R)$: $k(\alpha)=k(\phi(\alpha))$, and assume either both $R$ and $R'$ are reduced or both $R$ and $R'$ are non-reduced. Then $R$ is isomorphic to $R'$.
\end{prop}
\begin{proof}
We take $\{\alpha_{i}\}_{i=1}^{n}\subset cl(R)$ a set of simple roots of $B_{n}\subset B(0,n)$, and fix some $\{\alpha'_{i}\}_{i=1}^{n}\subset R$ such that for all $i:cl(\alpha'_{i})=\alpha_{i}$. Clearly, the set $\{\alpha'_{i}\}_{i=1}^{n}\sqcup\{\delta\}$ is a basis of $V$. Define $V_{1}:=span_{\mathbb{C}}\{\alpha'_{i}\}_{i=1}^{n}$ and denote $cl(R)=\{\alpha_{i}\}_{i=1}^{N}$, when $N>n$.

\par As $\{\alpha'_{i}\}_{i=1}^{n}\subset R$ generate a copy of $B_{n}$ we have that for all $i\in\{1,2,...,n,n+1,...,N\}$ such that $cl(\alpha'_{i})\neq\pm2\delta_{i}$ there exists $\alpha'_{i}\in V_{1}$ such that $cl(\alpha'_{i})=\alpha_{i}$. Note that these roots are obtained by simple reflections, and these reflections depend only on the classes $\{\alpha_{i}\}_{i=1}^{n}\subset cl(R)$ and not the representatives $\{\alpha'_{i}\}_{i=1}^{n}\subset R$ chosen. Thus we have for all such $i$s $cl^{-1}(\alpha_{i})=\{\alpha'_{i}+sk(\alpha_{i})\delta\}_{s\in\mathbb{Z}}$.

\par Finally, assume $R'\subset V'$ satisfies $\{\alpha_{i}\}_{i=1}^{N}=cl(R)\cong cl(R')=\{\phi(\alpha_{i})\}_{i=1}^{N}$. We observe the linear map from $V$ to $V'$ defined by $\alpha_{i}\mapsto\phi(\alpha_{i})'$ for all $i\in\{1,2,...,n\}$, $\delta\mapsto\delta'$ (note that now $R'=\{\phi(\alpha_{i})'+sk(\phi(\alpha_{i}))\delta'\}_{i=1,2,...,N}^{s\in\mathbb{Z}}$).
If both $R$ and $R'$ are reduced then by Lemma \ref{B0nlcomp} for all $cl(\alpha'_{i})=\delta_{i}$ one has $2\alpha'_{i}+\frac{k(2\delta_{i})}{2}\delta\in R$ and $\phi(2\delta_{i}+\frac{k(2\delta_{i})}{2}\delta)\in R'$. So this is an isomorphism of $R$ and $R'$. Similarly if both $R$ and $R'$ are not reduced then it is again an isomorphism, as then we may always choose these $\alpha'_{i}$s such that $2\alpha'_{i}$s are also roots.
\end{proof}

\subsubsection{}
\begin{cor}{}
Let $R$ be an AGRS such that $cl(R)\cong B(0,n)$, then $R$ is isomorphic to one of the following: $B(0,n)^{(1)},C(n+1)^{(2)},A_{2n}^{(2)},A(0,2n-1)^{(2)}$ and $A(0,2n)^{(4)}$.
\end{cor}
\begin{proof}
Recall that $A_{2n}^{(2)}$ is reduced and $A(0,2n-1)^{(2)}$ is non-reduced, and apply Lemma \ref{B0nlcomp} and Proposition \ref{b0nfinalprop}.
\end{proof}

\subsection{$cl(R)\cong A(n,n),n\geq2$} \

Let $cl(R)=A(n,n),n\geq2$, i.e. $cl(R)$ has the form (9) with $m=n\geq2$. If $R$ is finite then by Lemma \ref{lem2} $R\cong\tilde{A}(n,n)$. Assume $R$ is infinite.
\par We fix $\epsilon'_{i},\delta'_{i}\in V$ such that $\epsilon'_{i}-\epsilon'_{j},\delta'_{i}-\delta'_{j}\in R$ and $cl(\epsilon'_{i}-\epsilon'_{j})=\epsilon_{i}-\epsilon_{j},cl(\delta'_{i}-\delta'_{j})=\delta_{i}-\delta_{j}$. Recall there is only one $\widetilde{W}$ orbit and so $k$ is constant. It is easy to verify that $cl^{-1}(\{\epsilon_{i}-\epsilon_{j}\})$ and $cl^{-1}(\{\delta_{i}-\delta_{j}\})$ are both AGRSs with $cl(cl^{-1}(\{\epsilon_{i}-\epsilon_{j}\}))\cong cl(cl^{-1}(\{\delta_{i}-\delta_{j}\}))\cong A_{n}$. So we have $$cl^{-1}(\epsilon_{i}-\epsilon_{j})=\{\epsilon'_{i}-\epsilon'_{j}+\mathbb{Z}\delta\},cl^{-1}(\delta_{i}-\delta_{j})=\{\delta'_{i}-\delta'_{j}+\mathbb{Z}\delta\},$$
$$cl^{-1}(\{\epsilon_{i}-\epsilon_{j}\})\cong A_{n}^{(1)}\cong cl^{-1}(\{\delta_{i}-\delta_{j}\}).$$
Clearly there exists some $q\in\mathbb{C}$ such that $\epsilon'_{1}-\delta'_{1}+q\delta\in R$. By applying the generalized Weyl group action it is easy to see that for all $i,j$ one has $\epsilon'_{i}-\delta'_{j}+q\delta\in R$, and
$$cl^{-1}(\pm(\epsilon_{i}-\delta_{j}))=\{\pm(\epsilon'_{i}-\delta'_{j}+q\delta)+\mathbb{Z}\delta\}.$$
Thus we have $R\cong\tilde{A}(n,n)_{q}^{(1)}$. Verifying that $\tilde{A}(n,n)_{q}^{(1)}$ is indeed an AGRS for every $q\in\mathbb{C}$ is straightforward. We finish this part with the following lemma:

\subsubsection{}\label{easyiso}
\begin{lem}{}
$\tilde{A}(n,n)_{q}^{(1)}\cong\tilde{A}(n,n)_{q'}^{(1)}$ if and only if either $q-q'\in\mathbb{Z}$ or $q+q'\in\mathbb{Z}$.
\end{lem}
\begin{proof}
Let $\phi$ be an isomorphism of $\tilde{A}(n,n)_{q}^{(1)}$ and $\tilde{A}(n,n)_{q'}^{(1)}$:
$$\tilde{A}(n,n)_{q}^{(1)}:=\{\epsilon_{i}-\epsilon_{j}+\mathbb{Z}\delta,\delta_{i}-\delta_{j}+\mathbb{Z}\delta\}_{i\neq j}^{i,j=1,2,..n+1}\cup\{\pm(\epsilon_{i}-\delta_{j})+Id+(q+\mathbb{Z})\delta\}_{i=1,2,..,n+1}^{j=1,2,..,n+1},$$
$$\tilde{A}(n,n)_{q'}^{(1)}:=\{\epsilon_{i}-\epsilon_{j}+\mathbb{Z}\delta,\delta_{i}-\delta_{j}+\mathbb{Z}\delta\}_{i\neq j}^{i,j=1,2,..n+1}\cup\{\pm(\epsilon_{i}-\delta_{j})+Id+(q'+\mathbb{Z})\delta\}_{i=1,2,..,n+1}^{j=1,2,..,n+1}.$$

\par Observing $\phi(\epsilon_{1}-\epsilon_{2}+\mathbb{Z}\delta)$ it is clear that $\phi(\delta)\in\{\pm\delta\}$.

\par Clearly $\phi$ maps the non-isotropic roots to non-isotropic roots. If $\phi(\epsilon_{1}-\epsilon_{2})=\delta_{i}-\delta_{j}+l_{1,2}\delta$ for some $l_{1,2}\in\mathbb{Z}$ we compose $\phi$ with the automorphism $\epsilon_{i}\leftrightarrow\delta_{i},~1\leq i\leq n$, to get a new isomorphism satisfying $\tilde{\phi}(\delta)=\phi(\delta)$ and $\tilde{\phi}(\epsilon_{1}-\epsilon_{2})=\epsilon_{i}-\epsilon_{j}+l_{1,2}\delta$. So without loss of generality $\phi(\epsilon_{1}-\epsilon_{2})=\epsilon_{i}-\epsilon_{j}+l_{1,2}\delta$ for some $l_{1,2}\in\mathbb{Z}$.

\par Composing the automorphism $\epsilon_{1}\leftrightarrow\epsilon_{i},\epsilon_{2}\leftrightarrow\epsilon_{j}$ with $\phi$ we get a new isomorphism such that $\tilde{\phi}(\epsilon_{1}-\epsilon_{2})=\epsilon_{1}-\epsilon_{2}+l_{1,2}\delta$ (and of course still $\tilde{\phi}(\delta)=\phi(\delta)$). So without loss of generality $\phi(\epsilon_{1}-\epsilon_{2})=\epsilon_{1}-\epsilon_{2}+l_{1,2}\delta$.

Clearly, now there exists some $l_{2,3}\in\mathbb{Z}$ such that $\phi(\epsilon_{2}-\epsilon_{3})=\epsilon_{2}-\epsilon_{j'}+l_{2,3}\delta$, and $j'\neq 1,2$. Composing the automorphism $\epsilon_{3}\leftrightarrow\epsilon_{j'}$ with $\phi$ we get a new isomorphism such that $\tilde{\phi}(\epsilon_{2}-\epsilon_{3})=\epsilon_{2}-\epsilon_{3}+l_{2,3}\delta$ (and of course still the previous relations hold). So without loss of generality $\phi(\epsilon_{2}-\epsilon_{3})=\epsilon_{2}-\epsilon_{3}+l_{2,3}\delta$.

\par Continuing in this process we may assume that for all $1\leq i\leq n$ there exist $l_{i,i+1}\in\mathbb{Z}$ such that: $$\phi(\epsilon_{i}-\epsilon_{i+1})=\epsilon_{i}-\epsilon_{i+1}+l_{i,i+1}\delta,$$
and similarly there exist $s_{i,i+1}\in\mathbb{Z}$ such that: $$\phi(\delta_{i}-\delta_{i+1})=\delta_{i}-\delta_{i+1}+s_{i,i+1}\delta.$$
\par Then, it easily follows that there exists some $\kappa\in\mathbb{Z}$ such that: $$\phi(\epsilon_{n+1}-\delta_{1})=\epsilon_{n+1}-\delta_{1}+\kappa\delta.$$

\par As $\phi$ is a linear map:

$$0=\phi(Id+q\delta)=\phi(Id)+\phi(q\delta)=$$
$$\phi((\epsilon_{1}-\epsilon_{2})+2(\epsilon_{2}-\epsilon_{3})+..+n(\epsilon_{n}-\epsilon_{n+1})+(n+1)(\epsilon_{n+1}-\delta_{1})+n(\delta_{1}-\delta_{2})+..+(\delta_{n}-\delta_{n+1}))+\phi(q\delta)=$$
$$=\phi(\epsilon_{1}-\epsilon_{2})+2\phi(\epsilon_{2}-\epsilon_{3})+..+n\phi(\epsilon_{n}-\epsilon_{n+1})+(n+1)\phi(\epsilon_{n+1}-\delta_{1})+$$
$$+n\phi(\delta_{1}-\delta_{2})+(n-1)\phi(\delta_{2}-\delta_{3})+..+\phi(\delta_{n}-\delta_{n+1})+q\phi(\delta)=$$
$$=(\epsilon_{1}-\epsilon_{2})+l_{1,2}\delta+2(\epsilon_{2}-\epsilon_{3})+2l_{2,3}\delta+...+n(\epsilon_{n}-\epsilon_{n+1})+nl_{n,n+1}\delta+$$
$$+(n+1)(\epsilon_{n+1}-\delta_{1})+(n+1)\kappa\delta+n(\delta_{1}-\delta_{2})+ns_{1,2}\delta+..+(\delta_{n}-\delta_{n+1})+s_{n,n+1}\delta\pm q\delta=$$
$$=Id+(l_{1,2}+2l_{2,3}+..+nl_{n,n+1}+(n+1)\kappa+ns_{1,2}+..+s_{n,n+1})\delta\pm q\delta\in Id+\mathbb{Z}\delta\pm q\delta,$$
and so $q\pm q'\in\mathbb{Z}$.
\par For every $q,q'\in\mathbb{C}$ such that $q-q'\in\mathbb{Z}$ (resp. $q+q'\in\mathbb{Z}$) the map defined by: $$\forall1\leq i\leq n+1:\epsilon_{i}\mapsto\epsilon_{i},\delta_{i}\mapsto\delta_{i};\delta\mapsto\delta ~(\text {reps.}~ \delta\mapsto-\delta),$$
is an isomorphism of $\tilde{A}(n,n)_{q}^{(1)}$ and $\tilde{A}(n,n)_{q'}^{(1)}$.
\end{proof}

\subsection{$cl(R)$ is not a GRS} \

Let $cl(R)\cong C(m,n)$ or $cl(R)\cong BC(m,n)$, i.e. $cl(R)$ has one of the forms (16)-(17). We deal with the following two cases separately:

\subsubsection{$cl(R)\cong C(1,1)$} \

By corollary \ref{c11ends} either $R\cong\tilde{A}(1,1)$, and so is the root system of the Lie superalgebra $gl(2|2)$, or $R$ is a rational quotient of the set of real roots of the Lie superalgebra $gl(2|2)^{(1)}$.

\subsubsection{$cl(R)\ncong C(1,1)$} \

If $cl(R)\cong C(m,n)$, it is easy to see that $cl^{-1}(\{\pm2\epsilon_{i},\pm\epsilon_{i}\pm\epsilon_{j}\})$ and $cl^{-1}(\{\pm2\delta_{i},\pm\delta_{i}\pm\delta_{j}\})$ are both AGRSs with $cl(cl^{-1}(\{\pm2\epsilon_{i},\pm\epsilon_{i}\pm\epsilon_{j}\}))\cong C_{m}$, $cl(cl^{-1}(\{\pm2\delta_{i},\pm\delta_{i}\pm\delta_{j}\}))\cong C_{n}$. By the analysis done above $cl^{-1}(\{\pm2\epsilon_{i},\pm\epsilon_{i}\pm\epsilon_{j}\})\cong A_{2m-1}^{(2)}$, $cl^{-1}(\{\pm2\delta_{i},\pm\delta_{i}\pm\delta_{j}\})\cong A_{2n-1}^{(2)}$. We fix $\epsilon'_{i},\delta'_{i}\in V$ such that $2\epsilon'_{i},\epsilon'_{i}-\epsilon'_{j},2\delta'_{i},\delta'_{i}-\delta'_{j}\in R$. Clearly, there exists some $q\in\mathbb{C}$ (arbitrary up to $\mathbb{Z}$) such that $\epsilon'_{1}-\delta'_{1}+q\delta\in R$. By applying the generalized Weyl group action for every $i,j$ also $\pm\epsilon'_{i}\pm\delta'_{j}+q\delta\in R$. As $k(\epsilon_{1}-\delta_{1})=1$ and both $\epsilon'_{1}-\delta'_{1}+q\delta$ and $\epsilon'_{1}-\delta'_{1}-q\delta$ are roots, one has $2q\in\mathbb{Z}$, so without loss of generality $q\in\{0,\frac{1}{2}\}$. If $q=0$ then both $\epsilon'_{1}-\delta'_{1}+(\epsilon'_{1}+\delta'_{1})$ and $\epsilon'_{1}-\delta'_{1}-(\epsilon'_{1}+\delta'_{1})$ are roots, a contradiction. So $cl^{-1}(\epsilon_{i}\pm\delta_{j})=\{\epsilon'_{i}\pm\delta'_{j}+(\frac{1}{2}+\mathbb{Z})\delta\}$, and so $R\cong A(2m-1,2n-1)^{(2)}$.

If $cl(R)\cong BC(m,n)$, it is easy to see that $cl^{-1}(\{\pm2\epsilon_{i},\pm2\delta_{j},\pm\epsilon_{i}\pm\epsilon_{j},\pm\delta_{i}\pm\delta_{j},\pm\epsilon_{i}\pm\delta_{j}\})\cong A(2m-1,2n-1)^{(2)}$. So we may find $\epsilon'_{i},\delta'_{i}\in V$ such that $2\epsilon'_{i},\epsilon'_{i}-\epsilon'_{j},2\delta'_{i},\delta'_{i}-\delta'_{j},\epsilon'_{i}-\delta'_{j}+\frac{k(\epsilon_{i}+\epsilon_{j})}{2}\delta\in R$. There exists $q\in\mathbb{C}$ (arbitrary up to $k(\epsilon_{i})\delta$) such that $\epsilon'_{i}+q\delta\in R$. As $r_{2\epsilon'_{i}}(\epsilon'_{i}+q\delta)=-\epsilon'_{i}+q\delta$ one has $2q\in k(\epsilon_{i})\mathbb{Z}$ and so may assume $q\in\{0,\frac{k(\epsilon_{i})}{2}\}$.

\par If $k(\epsilon_{i})=\frac{k(\epsilon_{i}+\epsilon_{j})}{2}$, then $cl^{-1}(\{\pm2\epsilon_{i},\pm\epsilon_{i},\pm\epsilon_{i}\pm\epsilon_{j}\})\cong A(2m,0)^{(4)}$, and $cl^{-1}(\{\pm2\delta_{i},\pm\delta_{i},\pm\delta_{i}\pm\delta_{j}\})\cong A(0,2n)^{(4)}$. So $q=0$ and $R\cong A(2m,2n)^{(4)}$.

\par If $k(\epsilon_{i})=k(\epsilon_{i}+\epsilon_{j})$, then if $q=0$: $r_{\epsilon'_{i}-\delta'_{j}+\frac{1}{2}\delta}(\epsilon'_{i})=\delta'_{j}+\frac{1}{2}\delta$. Doing the same if $q=\frac{1}{2}$ (recall that $k(\epsilon_{i})=1$ in all cases) we get that in any case exactly one of $cl^{-1}(\{\pm2\epsilon_{i},\pm\epsilon_{i},\pm\epsilon_{i}\pm\epsilon_{j}\})$ and $cl^{-1}(\{\pm2\delta_{i},\pm\delta_{i},\pm\delta_{i}\pm\delta_{j}\})$ is reduced.
If $cl^{-1}(\{\pm2\epsilon_{i},\pm\epsilon_{i},\pm\epsilon_{i}\pm\epsilon_{j}\})$ is reduced then $R\cong A(2m,2n-1)^{(2)}$, and if $cl^{-1}(\{\pm2\delta_{i},\pm\delta_{i},\pm\delta_{i}\pm\delta_{j}\})$ is reduced then $R\cong A(2n,2m-1)^{(2)}$.

\newpage

\section{Root system -- Lie structure correspondence}\label{coranaly}
The following Table \ref{table1} summarizes all known correspondences between different types of irreducible root systems and classes of simple/indecomposible Lie structures. In all cases $V$ is a finite dimensional complex vector space with a symmetric bilinear form $(-,-)$. $R\subseteq V$ is a subset satisfying conditions (0)-(2) of Definition \ref{defroot1} if $R$ is either a Euclidean system or a GRS (first and third rows in the table). $R\subseteq V$ is a subset satisfying conditions (1)-(2) of Definition \ref{defroot1} and conditions (0'),(3),(4) of Definition \ref{affinegrs} if $R$ is either an ARS or an AGRS (second, fourth, fifth and sixth rows in the table).

\begin{table}[h]
\begin{center}
  \begin{tabular}{| c | c | c | c | c |}
    \hline
    The root system & $R$ is & $dimR^{\perp}$ & $R$ contains & Corresponding \\
    ($R$) type & finite &  & isotropic roots & Lie structure \\ \hline \hline
    Euclidean (RS) & yes & 0 & no & Finite dimensional algebras  \\
     &  &  &  & or $osp(1|2n)$ (which is a basic  \\
     &  &  &  & classical superalgebra) \\ \hline
    Non-Euclidean ; Non-isotropic & no & 1 & no & Symmetrizable affine KM \\
     (ARS) &  &  & & superalgebras that contain \\
     &  &  & & no isotropic roots$^\sharp$ \\ \hline
    GRS & yes & 0 & yes & Basic classical superalgebras \\
    (that are not RS) &  &  &  & with a non-trivial \\
     &  &  &  & odd part except $psl(2|2)$  \\
     &  &  &  & and $osp(1|2n)$ \\ \hline
    AGRS & no & 1 & yes & Symmetrizable affine \\
    with $cl(R)\ncong A(n,n)$&  &  &  & KM superalgebras that \\
    (that are not ARS)&  &  &  & contain isotropic roots \\
    & & & & except $gl(n|n)^{(1)},n\geq1$ \\ \hline
    Finite AGRS & yes & 1 & yes & $gl(n|n),n\geq2$ \\ \hline
    Infinite AGRS & no & 1 & yes & Rational quotients of  \\
    with $cl(R)\cong A(n,n)$& & & & $gl(2|2)^{(1)}$ or infinite quotients \\
    & & & & of $gl(n|n)^{(1)},n\geq3$ \\
    \hline
  \end{tabular}
    \caption{All known irreducible root systems -- simple/indecomposible Lie structures correspondences.}\label{table1}
\end{center}
\end{table}

\thanks{$^\sharp$
$B(0,n)^{(1)},C(n+1)^{(2)},A_{2n}^{(2)},A(0,2n-1)^{(2)}$ and $A(0,2n)^{(4)}$.

\begin{rem}{}
Assume that $V$ is a finite dimensional complex vector space with a non-degenerate symmetric bilinear form $(-,-)$, and that $R\subset V$ is an RS. Define a new form on $V$ by scaling $(-,-)$ by a non-zero scalar, i.e. $\widetilde{(-,-)}:=a(-,-)$ for some $a\in\mathbb{C}^{*}$. One easily sees that $R\subset V$ is an RS also when considering $V$ as a finite dimensional complex vector space with the non-degenerate symmetric bilinear form $\widetilde{(-,-)}$. Thus, the Euclidean systems indeed lie in the Euclidean space $span_{\mathbb{R}}\{R\}$, but not any restricted non-degenerate symmetric bilinear form from $V=span_{\mathbb{C}}\{R\}$ to $span_{\mathbb{R}}\{R\}$ gives a Euclidean structure. For an arbitrary non-degenerate symmetric bilinear form on $V$, a Euclidean structure is obtained by scaling the restricted form by a proper scalar. For example, $R=\{\pm\delta_{1},\pm\delta_{2},\pm\delta_{1}\pm\delta_{2}\}$ with the form $(\delta_{i},\delta_{j})=-\delta_{i,j}$: $R\cong B_{2}$ and in order to get a Euclidean structure we can scale the restricted form by $a=-1$.
\end{rem}

\section{Our approach as a generalization of Macdonald's}\label{macrem}

Ian G. Macdonald (in \cite{M1}) actually classified all root systems that do not contain isotropic roots. These are the systems of the first two types in Table \ref{table1} (RS and ARS).
\par In this paper we used some other definitions, that are more appropriate when applying the theory to the super case (i.e. when introducing isotropic roots). In the following we show that our definitions do generalize Macdonald's axioms (AR1)-(AR4) (see Section 2 in \cite{M1}).
\par First, one notices Macdonald works in a real vector space rather than complex. Let $R$ be a (possibly weak) GRS in a complex space, with no isotropic roots. We may scale one root $\alpha$ to satisfy $(\alpha,\alpha)=1$. Now, as for all $\beta\in R$ with $(\alpha,\beta)\neq0$ we have $\frac{2(\alpha,\beta)}{(\alpha,\alpha)}\in\mathbb{Z}$ we have $(\alpha,\beta)\in\mathbb{Q}$. As also $\frac{2(\alpha,\beta)}{(\beta,\beta)}\in\mathbb{Z}$ we have $(\beta,\beta)\in\mathbb{Q}$. Finally, as $R$ is irreducible we have $(\gamma,\gamma)\in\mathbb{Q}$ for all $\gamma\in R$ so $R$ spans also a real (even rational) vector space. Due to Serre (see \cite{Serre1}) this real space has a Euclidean structure. This argument does not hold in the presence of isotropic roots (unless the set of non-isotropic roots is irreducible) -- for example $D(2,1;\lambda)$ has 6 non-isotropic roots (3 copies of $A_{1}$) and it does not span a real space if $\lambda\not\in\mathbb{R}$.

\par In this context we note, that condition (4) in \ref{affinegrs} may be replaced by the following "discreteness condition" (4'):
\par (4') For every $\alpha\in R$ and $\delta\in ker(-,-)$: $|\{\alpha+s\delta\}_{0\leq s\leq1}\cap R|<\infty$.
\par Macdonald's proper action of the Weyl group condition (AR4) corresponds to this discreteness condition, and over $\mathbb{R}$ changing our condition (4) to (4') would lead to an equivalent definition of an AGRS whenever $cl(R)\ncong C(1,1)$ (just apply reasoning as in the proofs of Lemmas \ref{nonisostr} and \ref{lem3}). When $cl(R)\cong C(1,1)$ (4') also allows other systems that are not allowed by (4), such as $C(1,1)^{q}$ for any $q\in\mathbb{C}\setminus\mathbb{Q}$ satisfying $0<|real(q)|<1$ (these systems correspond to infinite non-rational quotients of $gl(2|2)^{(1)}$). Anyhow, in the absence of isotropic roots (4) and (4') are equivalent.

\par One also notes that Macdonald has no finiteness condition. By Lemma \ref{finitenessofgrs2} below, in a Euclidean space (which is the setting of Macdonald) finiteness follows from the other conditions.

\subsection{}\label{finitenessofgrs2}
\begin{lem}{}
Let $R$ be a set satisfying all conditions of a (possibly weak) GRS in $V$, except being finite, and $V$ is a Euclidean space (i.e. a real space and not complex, with the form $(-,-)$ is positive definite). Then $R$ is finite.
\end{lem}
\begin{proof}
Let $\{\alpha_{i}\}_{i=1}^{n}$ be a basis of $V$ consists of roots from $R$. Let $\beta\in R$. For every $i=1,2,...,n$ the two roots $\alpha_{i},\beta$ span either a 1 or a 2 dimensional Euclidean space, so we have: $$\frac{2(\alpha_{i},\beta)}{(\alpha_{i},\alpha_{i})},\frac{2(\alpha_{i},\beta)}{(\beta,\beta)}\in\mathbb{Z}, |\frac{(\alpha_{i},\beta)}{(\alpha_{i},\alpha_{i})}\frac{(\alpha_{i},\beta)}{(\beta,\beta)}|=|\cos^{2}(\theta)|\leq1,$$ when $\theta$ is the angle between $\alpha_{i}$ and $\beta$, and so $\frac{2(\alpha_{i},\beta)}{(\alpha_{i},\alpha_{i})}\in\{0,\pm1,\pm2,\pm3\pm4\}$. Clearly the set $\{\frac{2(\alpha_{i},\beta)}{(\alpha_{i},\alpha_{i})}\}_{i=1}^{n}$ completely determines $\beta$. As there are only finitely many options for this set, there are only finitely many roots in $R$, thus $R$ is finite.
\end{proof}

\section{General AGRSs}\label{simple}

\par AGRSs that are not irreducible arise naturally, for instance when taking the non-isotropic part of an irreducible AGRS. Above we classified all irreducible AGRSs. In the following Theorem \ref{finunion} we show that any AGRS $R$ is a disjoint union of finitely many irreducible AGRSs $I_{i}$ and GRSs $F_{i}$, and that the space spanned by $R$ is isomorphic to a direct sum of the spaces spanned by all $F_{i}$, the spaces spanned by the weak GRSs $cl(I_{i})$ and a one dimensional space (corresponding to $ker(-,-)$).

\subsection{}\label{finunion}
\begin{thm}{}
Let $V$ be a complex vector space with a symmetric bilinear form $(-,-)$, and let $R\subset V$ be an AGRS, then:
$$R=\sqcup_{i\in J_{1}} I_{i}\sqcup_{i\in J_{2}} F_{i},$$
$$V\cong\mathbb{C}\oplus_{i\in J_{1}} (V_{i}/ker(-,-)_{i}) \oplus_{i\in J_{2}}V_{i},$$
when $|J_{1}\cup J_{2}|<\infty,J_{1}\cap J_{2}=\emptyset,J_{1}\neq\emptyset$, for every $i$: $V_{i}$ is a complex vector space with a symmetric bilinear form $(-,-)_{i}$, $I_{i}\subset V_{i}$ is an irreducible AGRS, and $F_{i}\subset V_{i}$ is an irreducible GRS.
\end{thm}
\begin{proof}
Let $R\subset V$ be an AGRS. As $cl(R)$ is a weak GRS it is a finite union of irreducible weak GRSs: $cl(R)=\sqcup_{i=1}^{n} S_{i}$ when $S_{i}\subset cl(R)$. We define $V'_{i}:=span_{\mathbb{C}}\{S_{i}\}$. For every $i$ we choose a basis $\{\tilde{\alpha}_{j}^{i}\}_{j=1}^{d_{i}}$ of $V'_{i}$, containing elements of $S_{i}$ only (clearly such basis exists), and choose roots $\{\alpha_{j}^{i}\}_{j=1}^{d_{i}}$ in $R$ such that $cl(\alpha_{j}^{i})=\tilde{\alpha}_{j}^{i}$. We identify $\{\alpha_{j}^{i}\}_{j=1}^{d_{i}}$ with $S_{i}$, and so $\{\{\alpha_{j}^{i}\}_{j=1}^{d_{i}}\}_{i=1}^{n}$ with $cl(R)$. Denote $V_{i}:=span_{\mathbb{C}}\{\cup_{j=1}^{d_{i}}(cl^{-1}(\alpha_{j}^{i}))\}$ and $(-,-)_{i}:=(-,-)|_{V_{i}}$. Clearly, $span_{\mathbb{C}}\{\{\alpha_{j}^{i}\}_{j=1}^{d_{i}}\}_{i=1}^{n}$ is a linearly independent set in $R$, and $ker(-,-)\cap span_{\mathbb{C}}\{\{\alpha_{j}^{i}\}_{j=1}^{d_{i}}\}_{i=1}^{n}=\emptyset$. Recall $dim(V)-dim(cl(R))=dim(ker(-,-))=1$, and so $V=ker(-,-)\oplus_{i=1}^{n}(span_{\mathbb{C}}\{\alpha_{j}^{i}\}_{j=1}^{d_{i}}\})$. We define $J_{1}=\{1\leq i \leq n|~cl^{-1}(S_{i})~\text {is not a GRS}\}$, and $J_{2}=\{1,2,...,n\}\setminus J_{1}$. Clearly $|J_{1}\cup J_{2}|<\infty,J_{1}\cap J_{2}=\emptyset$ and as the form $(-,-)$ is degenerate $J_{1}\neq\emptyset$.
For all $i\in J_{1}$: $V_{i}/ker(-,-)\cong V'_{i}\cong span_{\mathbb{C}}\{\alpha_{j}^{i}\}_{j=1}^{d_{i}}$ and for all $i\in J_{2}$: $V_{i}\cong span_{\mathbb{C}}\{\alpha_{j}^{i}\}_{j=1}^{d_{i}}$, and so $V$ has the required form.
\end{proof}

\section{Further remarks}\label{further}

\subsection{Even and odd roots} \

One notices that in our definition of AGRS, as well as in Serganova's definition of GRS, a-priori $R$ is not a disjoint union of two parts $R_{\ol{0}}$ and $R_{\ol{1}}$, that represent the even and odd parts of the corresponding superalgebra respectively. It is possible to get this partition, by defining $R_{\ol{1}}=\{\alpha\in R|(\alpha,\alpha)=0 \text { or } 2\alpha\in R\}$, and $R_{\ol{0}}=R\setminus R_{\ol{1}}$.
\par Alternatively, we may define parity in the following way: we call a map $f:R\to\mathbb{Z}_{2}$ {\em linear} if for all $\alpha,\beta\in R$ such that $\alpha+\beta\in R: f(\alpha)+f(\beta)=f(\alpha+\beta)$. A linear map $f$ that maps all isotropic roots to $\ol{1}$ is called a {\em parity function}, and we define for such $f$ odd (resp. even) roots by $R_{\ol{1}}^{f}=f^{-1}(\ol{1})$ ($R_{\ol{0}}^{f}=f^{-1}(\ol{0})$). Clearly the first definition is a special case of the second (with $f(\alpha)=\ol{1}$ if $(\alpha,\alpha)=0 \text { or } 2\alpha\in R$ and $f(\alpha)=\ol{0}$ otherwise), but note that the second is not always unique.
\par As an example, take some parity function $f$ of $A(m,n)$. By definition $f(\pm(\epsilon_{i}-\delta_{j}))=\ol{1}$. As $f$ is linear $f(\epsilon_{i}-\epsilon_{k})=f((\epsilon_{i}-\delta_{j})+(\delta_{j}-\epsilon_{k}))=f(\epsilon_{i}-\delta_{j})+f(\delta_{j}-\epsilon_{k})=\ol{1}+\ol{1}=\ol{0}$, and the same for $\delta_{i}-\delta_{j}$. Thus $f$ is uniquely defined (and of course corresponds to the first definition). On the other hand we have two different parity functions of $B(0,n)$. The first is the one corresponding to the first definition ($\delta_{i}\mapsto\ol{1},2\delta_{i},\delta_{i}\pm\delta_{j}\mapsto\ol{0}$) and the second is the trivial map ($B(0,n)\mapsto\ol{0}$). One easily sees that these are the only parity functions of $B(0,n)$.

\subsection{Root subsystems}
\subsubsection{}
\begin{defn}{}
Let $R\subset V$ be a root system (of any type). We call a subset $S\subseteq R$ a {\em root subsystem}, if $S\subset span_{\mathbb{C}}\{S\}$ is a root system (of any type).
\end{defn}
\par Note that Theorem \ref{thm1} implies:
\subsubsection{}
\begin{cor}{}
Let $R$ be the set of real roots of a root system of a symmetrizable affine indecomposable Kac-Moody superalgebra, and let $S\subset R$ be an irreducible root subsystem. Then $S$ is itself one of the following:
\par (1) The set of real roots of either a finite dimensional Lie superalgebra (that is simple unless it is $gl(n|n),~n\geq2$), or a symmetrizable affine indecomposable Kac-Moody superalgebra.
\par (2) A rational quotient of $gl(2|2)^{(1)}$.
\par (3) An infinite quotient of $gl(n|n)^{(1)},~n\geq3$.
\end{cor}
\par When dealing with subsystems we notice two interesting facts about parity of roots: let us consider parity as in the first definition (as said, then it is unique). First, we see that parity is not invariant with respect to the restriction from a root system to a subsystem. For example, all roots of the form $\pm\delta_{i}$ are odd in $B(0,n)$, but when one thinks about these as roots in the subsystem $B_{n}\subset B(0,n)$ they are all even. Second, in most cases two roots $\alpha,\beta$ with $cl(\alpha)=cl(\beta)$ have the same parity. We see a nice example that it is not so, when taking $A(0,2n)^{(4)}$ as a root subsystem of $B(m,n)^{(1)}$. As before some odd roots in $B(m,n)^{(1)}$ (the roots $\{\delta_{i}+2s\delta\}_{s\in\mathbb{Z}}$) become even in $A(0,2n)^{(4)}$. Moreover, in $A(0,2n)^{(4)}$ all roots $\delta_{i}$ are even, but $\delta_{i}+\delta$ are odd. In $B(m,n)^{(1)}$ two roots $\alpha,\beta$ with $cl(\alpha)=cl(\beta)$ have the same parity.

\subsection{The missing basic classical Lie superalgebra in Serganova's classification}\label{sergexception}
The only absent set of real roots of the root systems of basic classical Lie superalgebras from Serganova's classification of GRSs is $A(1,1)$, which corresponds to the superalgebra $psl(2|2)=sl(2|2)/<\mathbb{1}_{4}>$. Interestingly, $A(1,1)\cong C(1,1)\cong cl(\tilde{A}(1,1))$, i.e. it is the image of the finite AGRS $\tilde{A}(1,1)$ under the quotient by the kernel of the bilinear form. This is not a GRS, but it is a weak GRS. Note that unlike the cases $cl:\tilde{A}(n,n)\to A(n,n),~n\geq2$, the map $cl$ is not bijective: the pre-image of any non-isotropic root (respectively isotropic root) has cardinality 1 (2). In that context see Serganova's third alternative definition for GRS (not involving inner product) and Theorem 7.2 in \cite{VGRS}.

\subsection{Observation on Kashiwara-Tanisaki's approach} \

M. Kashiwara and T. Tanisaki proved (see \cite{KT1}) that if $\Delta$ is the set of real roots of a semi-simple finite or affine Lie algebra, then any subsystem $\Delta'\subseteq\Delta$ is itself the set of real roots of a semi-simple finite or affine Lie algebra (this result easily follows from Macdonald's classification). Their approach is based on applying the Weyl group action on $\Delta'$.
\par In the context of Lie superalgebras, the classical Weyl group is generated by all even reflections: $W:=<r_{\alpha}>_{\alpha\in\Delta_{\ol{0}}}$ (that is exactly the group generated by all non-odd reflections). Our first idea was to generalize \cite{KT1}'s approach by applying the action of $\widetilde{W}$ on $\Delta'$. This generalization attempt was unsuccessful, as the action of this group on the set $\Delta$ does not map root subsystems to root subsystems, i.e. $\Delta'\subseteq\Delta$ is a root subsystem and $\omega\in\widetilde{W}$ does not imply $\omega(\Delta')\subseteq\Delta$ is a root subsystem. \par A counter example can be found even in finite systems: take $B_{2}:=\{\pm\epsilon_{1},\pm\epsilon_{2},\pm\epsilon_{1}\pm\epsilon_{2}\}\subseteq B(2,1)$ and take $\alpha=\epsilon_{2}-\delta_{1}\in B(2,1)$. Then $r_{\alpha}(\epsilon_{1}\pm\epsilon_{2})=\epsilon_{1}\pm\delta_1$, but
$r_{\epsilon_{1}-\delta_{1}}(\epsilon_{1}+\delta_{1})=2\delta_{1}\not\in r_{\alpha}(B_{2})=\{\pm\epsilon_{1},\pm\delta_{1},\pm\epsilon_{1}\pm\delta_1\}$, and so $r_{\alpha}(B_{2})$ is not a root system of any type.

\subsection{Further generalizations} \

Relaxing the conditions on $V$ and $R$ may lead to classifying further types of systems corresponding to other Lie structures. The two natural guesses would be either allowing $dimR^{\perp}\geq2$ or allowing accumulation points in the pre-image $cl^{-1}(\alpha)$ for some (which would probably imply for all) $\alpha\in R$.

\end{document}